\documentclass[12pt,reqno,openbib,runningheads,a4paper]{amsart}%
\usepackage{eurosym}
\usepackage{amssymb}
\usepackage{amsfonts}
\usepackage{amsmath}
\usepackage{graphicx}
\usepackage[authoryear]{natbib}
\usepackage[dvipsnames]{xcolor}
\usepackage[legalpaper,bookmarks=true,colorlinks=true,linkcolor=blue,citecolor=blue]%
{hyperref}%
\providecommand{\U}[1]{\protect\rule{.1in}{.1in}}
\providecommand{\U}[1]{\protect\rule{.1in}{.1in}}
\newtheorem{theorem}{Theorem}[section]

\newtheorem{lemma}{Lemma}[section]

\makeatletter
\renewcommand{\@biblabel}[1]{}
\makeatother
\begin{document}

\begin{center}
{\Large \textbf{Estimating the second-order parameter of regular variation and
bias reduction in tail index estimation under random truncation}}

\medskip\medskip

{\large Nawel Haouas, Abdelhakim Necir}$^{\ast},$ {\large Brahim
Brahimi}\medskip

{\small \textit{Laboratory of Applied Mathematics, Mohamed Khider University,
Biskra, Algeria}}\medskip\medskip%
\[
\]

\end{center}

\noindent\textbf{Abstract}\medskip

\noindent In this paper, we propose an estimator of the second-order parameter
of randomly right-truncated Pareto-type distributions data and establish its
consistency and asymptotic normality. Moreover, we derive an asymptotically
unbiased estimator of the tail index and study its asymptotic behaviour. Our
considerations are based on a useful Gaussian approximation of the tail
product-limit process recently given by Benchaira et al. [Tail product-limit
process for truncated data with application to extreme value index estimation.
\textit{Extremes, }2016; 19: 219-251] and the results of Gomes et al.
[Semi-parametric estimation of the second order parameter in statistics of
extremes. \textit{Extremes, }2002\textit{;} 5: 387-414].\ We show, by
simulation, that the proposed estimators behave well, in terms of bias and
mean square error.\medskip

\noindent\textbf{Keywords:} Bias-reduction; Extreme value index; Product-limit
estimator; Random truncation; Second-order parameter.\medskip

\noindent\textbf{AMS 2010 Subject Classification:} 60F17, 62G30, 62G32, 62P05.

\vfill

\vfill

\noindent{\small $^{\text{*}}$Corresponding author:
\texttt{necirabdelhakim@yahoo.fr} \newline}

\noindent{\small \textit{E-mail addresses:}\newline%
\texttt{nawel.haouas@yahoo.com} (N.~Haoues)\newline%
\texttt{brah.brahim@gmail.com} (B.~Brahimi)}

\section{\textbf{Introduction\label{sec1}}}

\noindent Let $\left(  \mathbf{X}_{i},\mathbf{Y}_{i}\right)  ,$ $1\leq i\leq
N$ be a sample of size $N\geq1$ from a couple $\left(  \mathbf{X}%
,\mathbf{Y}\right)  $ of independent random variables (rv's) defined over some
probability space $\left(  \Omega,\mathcal{A},\mathbf{P}\right)  ,$ with
continuous marginal distribution functions (df's) $\mathbf{F}$ and
$\mathbf{G}$ respectively.$\ $Suppose that $\mathbf{X}$ is truncated to the
right by $\mathbf{Y},$ in the sense that $\mathbf{X}_{i}$ is only observed
when $\mathbf{X}_{i}\leq\mathbf{Y}_{i}.$ We assume that both survival
functions $\overline{\mathbf{F}}:=1-\mathbf{F}$\textbf{\ }and $\overline
{\mathbf{G}}:=1-\mathbf{G}$ are regularly varying at infinity with negative
tail indices $-1/\gamma_{1}$ and $-1/\gamma_{2}$ respectively, that is, for
any $x>0$%
\begin{equation}
\lim_{z\rightarrow\infty}\frac{\overline{\mathbf{F}}\left(  xz\right)
}{\overline{\mathbf{F}}\left(  z\right)  }=x^{-1/\gamma_{1}}\text{ and }%
\lim_{z\rightarrow\infty}\frac{\overline{\mathbf{G}}\left(  xz\right)
}{\overline{\mathbf{G}}\left(  z\right)  }=x^{-1/\gamma_{2}}. \label{RV-1}%
\end{equation}
Since the weak approximations of extreme value theory based statistics are
achieved in the second-order framework \citep[see][]{deHS96}, then it seems
quite natural to suppose that both df's $\mathbf{F}$ and $\mathbf{G}$ satisfy
the well-known second-order condition of regular variation specifying the
rates of convergence in $(\ref{RV-1}).$ That is, we assume that for any $x>0$%
\begin{equation}
\underset{t\rightarrow\infty}{\lim}\dfrac{\mathbb{U}_{\mathbf{F}}\left(
tx\right)  /\mathbb{U}_{\mathbf{F}}\left(  t\right)  -x^{\gamma_{1}}%
}{\mathbf{A}_{\mathbf{F}}\left(  t\right)  }=x^{\gamma_{1}}\dfrac{x^{\rho_{1}%
}-1}{\rho_{1}}, \label{second-order}%
\end{equation}
and%
\begin{equation}
\underset{t\rightarrow\infty}{\lim}\dfrac{\mathbb{U}_{\mathbf{G}}\left(
tx\right)  /\mathbb{U}_{\mathbf{G}}\left(  t\right)  -x^{\gamma_{2}}%
}{\mathbf{A}_{\mathbf{G}}\left(  t\right)  }=x^{\gamma_{2}}\dfrac{x^{\rho_{2}%
}-1}{\rho_{2}}, \label{second-orderG}%
\end{equation}
where $\rho_{1},\rho_{2}<0$\ are the second-order parameters and
$\mathbf{A}_{\mathbf{F}},$ $\mathbf{A}_{\mathbf{G}}$\ are functions tending to
zero and not changing signs near infinity with regularly varying absolute
values at infinity with indices $\rho_{1},$ $\rho_{2}$ respectively. For any
df $H,$ $\mathbb{U}_{H}\left(  t\right)  :=H^{\leftarrow}\left(  1-1/t\right)
,$ $t>1,$ stands for the quantile function. This class of distributions, which
includes models such as Burr, Fr\'{e}chet, Generalized Pareto,
Student,log-gamma, stable,... takes a prominent role in extreme value theory.
Also known as heavy-tailed, Pareto-type or Pareto-like distributions, these
models have important practical applications and are used rather
systematically in certain branches of non-life insurance, as well as in
finance, telecommunications, hydrology, etc... \citep[see, e.g.,][]{Res06}. We
denote the observed observations of the truncated sample $\left(
\mathbf{X}_{i},\mathbf{Y}_{i}\right)  ,$ $i=1,...,N,$ by $\left(  X_{i}%
,Y_{i}\right)  ,$ $i=1,...,n,$ as copies of a couple of rv's $\left(
X,Y\right)  ,$ where $n=n_{N}$ is a sequence of discrete rv's for which,\ by
of the weak law of large numbers satisfies $n_{N}/N\overset{\mathbf{P}%
}{\rightarrow}p:=\mathbf{P}\left(  \mathbf{X}\leq\mathbf{Y}\right)  ,$ as
$N\rightarrow\infty.$ The usefulness of the statistical analysis under random
truncation is shown in \cite{Hrbst-99} where the authors applies truncated
model techniques to estimate loss reserves for incurred but not reported
(IBNR) claim amounts.\textbf{\ }For a recent discussion on randomly
right-truncated insurance claims, one refers to \cite{EO-2008}. In
reliability, a real dataset, consisting in lifetimes of automobile brake pads
and already considered by \cite{Lawless} in page 69, was recently analyzed in
\cite{GS2015} and \cite{BchMN-16a} as an application of randomly truncated
heavy-tailed models. The joint distribution of $X_{i}$ and $Y_{i}$ is
$H\left(  x,y\right)  :=\mathbf{P}\left(  X\leq x,Y\leq y\right)
=\mathbf{P}\left(  \mathbf{X}\leq\min\left(  x,\mathbf{Y}\right)
,\mathbf{Y}\leq y\mid\mathbf{X}\leq\mathbf{Y}\right)  ,$ which equals $p^{-1}%
%TCIMACRO{\dint _{0}^{y}}%
%BeginExpansion
{\displaystyle\int_{0}^{y}}
%EndExpansion
\mathbf{F}\left(  \min\left(  x,z\right)  \right)  d\mathbf{G}\left(
z\right)  .$ The marginal distributions of the rv's $X$ and $Y,$ respectively
denoted by $F^{\ast}$ and $G^{\ast},$ are equal to $p^{-1}%
%TCIMACRO{\dint _{0}^{x}}%
%BeginExpansion
{\displaystyle\int_{0}^{x}}
%EndExpansion
\overline{\mathbf{G}}\left(  z\right)  d\mathbf{F}\left(  z\right)  $ and
$p^{-1}\int_{0}^{y}\mathbf{F}\left(  z\right)  d\mathbf{G}\left(  z\right)  ,$
respectively. The tail of df $F^{\ast}$ simultaneously depends on
$\overline{\mathbf{G}}$ and $\overline{\mathbf{F}}$ while that of $G^{\ast}$
only relies on $\overline{\mathbf{G}}\mathbf{.}$ By using Proposition B.1.10
in \cite{deHF06}, to the regularly varying functions $\overline{\mathbf{F}}$
and $\overline{\mathbf{G}},$ we also show that both $\overline{G}^{\ast}$ and
$\overline{F}^{\ast}$ are regularly varying at infinity, with respective
indices $\gamma_{2}$ and $\gamma:=\gamma_{1}\gamma_{2}/\left(  \gamma
_{1}+\gamma_{2}\right)  .$ Recently \cite{GS2015} addressed the estimation of
the extreme value index $\gamma_{1}$ under random right-truncation and used
the definition of $\gamma$ to derive a consistent estimator as a quotient of
two Hill estimators \citep{Hill75} of tail indices $\gamma$ and $\gamma_{2}$
which are based on the upper order statistics $X_{n-k:n}\leq...\leq X_{n:n}$
and $Y_{n-k:n}\leq...\leq Y_{n:n}$ pertaining to the samples $\left(
X_{1},...,X_{n}\right)  $ and $\left(  Y_{1},...,Y_{n}\right)  $ respectively.
The sample fraction $k=k_{n}$ being a (random) sequence of integers such that,
given $n=m=m_{N},$ $k_{m}\rightarrow\infty$ and $k_{m}/m\rightarrow0$ as
$N\rightarrow\infty.$ Under the tail dependence and the second-order regular
variation conditions, \cite{BchMN-15} established the asymptotic normality of
this estimator. Recently, \cite{WW2016} proposed an asymptotically normal
estimator for $\gamma_{1}$ by considering a Lynden-Bell integration with a
deterministic threshold. The case of a random threshold, is addressed by
\cite{BchMN-16a} who propose a Hill-type estimator for randomly
right-truncated data, defined by%
\begin{equation}
\widehat{\gamma}_{1}^{\left(  BMN\right)  }=\sum_{i=1}^{k}a_{n}^{\left(
i\right)  }\log\frac{X_{n-i+1:n}}{X_{n-k:n}}, \label{BMN}%
\end{equation}
where%
\[
a_{n}^{\left(  i\right)  }:=\dfrac{\mathbf{F}_{n}\left(  X_{n-i+1:n}\right)
/C_{n}\left(  X_{n-i+1:n}\right)  }{%
%TCIMACRO{\dsum \limits_{i=1}^{k}}%
%BeginExpansion
{\displaystyle\sum\limits_{i=1}^{k}}
%EndExpansion
\mathbf{F}_{n}\left(  X_{n-i+1:n}\right)  /C_{n}\left(  X_{n-i+1:n}\right)
},
\]
with $\mathbf{F}_{n}\left(  x\right)  :=\prod_{i:X_{i}>x}\exp\left\{
-\dfrac{1}{nC_{n}\left(  X_{i}\right)  }\right\}  $ is the well-known
product-limit Woodroofe's estimator \citep{W-85} of the underlying df
$\mathbf{F}$ and $C_{n}\left(  x\right)  :=n^{-1}\sum\nolimits_{i=1}%
^{n}\mathbf{1}\left(  X_{i}\leq x\leq Y_{i}\right)  .$ The authors show by
simulation that, for small datasets, their estimator behaves better in terms
of bias and root of the mean squared error (rmse), than Gardes-Supfler's
estimator. Moreover, they establish the asymptotic normality by considering
the second-order regular variation conditions $(\ref{second-order})$\ and
$(\ref{second-orderG})$ and the assumption $\gamma_{1}<\gamma_{2}.$ More
precisely, they show that, for a sufficiently large $N,$%
\begin{equation}
\widehat{\gamma}_{1}^{\left(  BMN\right)  }=\gamma_{1}+k^{-1/2}\Lambda\left(
\mathbf{W}\right)  +\frac{\mathbf{A}_{0}\left(  n/k\right)  }{1-\rho_{1}%
}\left(  1+o_{\mathbf{P}}\left(  1\right)  \right)  , \label{approxi}%
\end{equation}
where $\mathbf{A}_{0}\left(  t\right)  :=\mathbf{A}_{\mathbf{F}}\left(
1/\overline{\mathbf{F}}\left(  \mathbb{U}_{F^{\ast}}\left(  t\right)  \right)
\right)  ,$ $t>1,$ and $\Lambda\left(  \mathbf{W}\right)  $ is a centred
Gaussian rv defined by%
\[
\Lambda\left(  \mathbf{W}\right)  :=\frac{\gamma}{\gamma_{1}+\gamma_{2}}%
\int_{0}^{1}\left(  \gamma_{2}-\gamma_{1}-\gamma\log s\right)  s^{-\gamma
/\gamma_{2}-1}\mathbf{W}\left(  s\right)  ds-\gamma\mathbf{W}\left(  1\right)
,
\]
with $\left\{  \mathbf{W}\left(  s\right)  ;\text{ }s\geq0\right\}  $ being a
standard Wiener process\ defined on the probability space $\left(
\Omega,\mathcal{A},\mathbf{P}\right)  .$ Thereby, they conclude that $\sqrt
{k}\left(  \widehat{\gamma}_{1}^{\left(  BMN\right)  }-\gamma_{1}\right)
\overset{\mathcal{D}}{\rightarrow}\mathcal{N}\left(  \lambda/\left(
1-\rho_{1}\right)  ,\sigma^{2}\right)  ,$ as $N\rightarrow\infty,$ where
$\sigma^{2}:=\gamma^{2}\left(  1+\gamma_{1}/\gamma_{2}\right)  \left(
1+\left(  \gamma_{1}/\gamma_{2}\right)  ^{2}\right)  /\left(  1-\gamma
_{1}/\gamma_{2}\right)  ,$ provided that, given $n=m,$ $\sqrt{k_{m}}%
\mathbf{A}_{0}\left(  m/k_{m}\right)  \rightarrow\lambda<\infty.$ Recently,
\cite{BchMN-16b} adopted the same approach to introduce a kernel estimator to
the tail index $\gamma_{1}$ which improves the bias of $\widehat{\gamma}%
_{1}^{\left(  BMN\right)  }.$ It is worth mentioning that the assumption
$\gamma_{1}<\gamma_{2}$ is required in order to ensure that it remains enough
extreme data for the inference to be accurate. In other words, they consider
the situation where the tail of the rv of interest is not too contaminated by
the truncation rv.$\mathbf{\medskip}$

\noindent The aim of this paper is the estimation of the second
order-parameter $\rho_{1}$ given in condition $(\ref{second-order})$ which, to
our knowledge, is not addressed yet in the extreme value literature. This
parameter is of practical relevance in extreme value analysis due its crucial
importance in selecting the optimal number of upper order statistics $k$ in
tail index estimation (see, e.g., \citeauthor{deHF06}, \citeyear[page
77]{deHF06}) and in reducing the bias of such estimation. In the case of
complete data, this problem has received a lot of attention from many authors
like, for instance, \cite{peng98}, \cite{FAHL03}, \cite{GHP2003}, \cite{PQ-4},
\cite{GBW2010}, \cite{WGG2012}, \cite{WW012}, \cite{DGG013}. Inspired by the
paper of \cite{GHP2003}, we propose an estimator for $\rho_{1}$ adapted to the
random right-truncation case. To this end, for $\alpha>0$ and $t>0,$ we
introduce the following tail functionals%
\begin{equation}
M^{\left(  \alpha\right)  }\left(  t;\mathbf{F}\right)  :=\frac{1}%
{\overline{\mathbf{F}}\left(  \mathbb{U}_{F^{\ast}}\left(  t\right)  \right)
}\int_{\mathbb{U}_{F^{\ast}}\left(  t\right)  }^{\infty}\log^{\alpha}\left(
x/\mathbb{U}_{F^{\ast}}\left(  t\right)  \right)  d\mathbf{F}\left(  x\right)
, \label{Mt}%
\end{equation}%
\begin{equation}
Q^{\left(  \alpha\right)  }\left(  t;\mathbf{F}\right)  :=\frac{M^{\left(
\alpha\right)  }\left(  t;\mathbf{F}\right)  -\Gamma\left(  \alpha+1\right)
\left(  M^{\left(  1\right)  }\left(  t;\mathbf{F}\right)  \right)  ^{\alpha}%
}{M^{\left(  2\right)  }\left(  t;\mathbf{F}\right)  -2\left(  M^{\left(
1\right)  }\left(  t;\mathbf{F}\right)  \right)  ^{2}}, \label{Q}%
\end{equation}
and%
\begin{equation}
S^{\left(  \alpha\right)  }\left(  t;\mathbf{F}\right)  :=\delta\left(
\alpha\right)  \frac{Q^{\left(  2\alpha\right)  }\left(  t;\mathbf{F}\right)
}{\left(  Q^{\left(  \alpha+1\right)  }\left(  t;\mathbf{F}\right)  \right)
^{2}}, \label{S}%
\end{equation}
where $\log^{\alpha}x:=\left(  \log x\right)  ^{\alpha}$ and $\delta\left(
\alpha\right)  :=\alpha\left(  \alpha+1\right)  ^{2}\Gamma^{2}\left(
\alpha\right)  /\left(  4\Gamma\left(  2\alpha\right)  \right)  ,$ with
$\Gamma\left(  \cdot\right)  $ standing for the usual Gamma function. From
assertion $\left(  ii\right)  $ of Lemma $\ref{lemma1},$ we have, for any
$\alpha>0,$%
\begin{equation}
M^{\left(  \alpha\right)  }\left(  t;\mathbf{F}\right)  \rightarrow\gamma
_{1}^{\alpha}\Gamma\left(  \alpha+1\right)  ,\text{ }Q^{\left(  \alpha\right)
}\left(  t;\mathbf{F}\right)  \rightarrow q_{\alpha}\left(  \rho_{1}\right)
\text{ and }S^{\left(  \alpha\right)  }\left(  t;\mathbf{F}\right)
\rightarrow s_{\alpha}\left(  \rho_{1}\right)  , \label{Slim}%
\end{equation}
as $t\rightarrow\infty,$ where%
\begin{equation}
q_{\alpha}\left(  \rho_{1}\right)  :=\frac{\gamma_{1}^{\alpha-2}\Gamma\left(
\alpha+1\right)  \left(  1-\left(  1-\rho_{1}\right)  ^{\alpha}-\alpha\rho
_{1}\left(  1-\rho_{1}\right)  ^{\alpha-1}\right)  }{\rho_{1}^{2}\left(
1-\rho_{1}\right)  ^{\alpha-2}}, \label{q}%
\end{equation}
and%
\begin{equation}
s_{\alpha}\left(  \rho_{1}\right)  :=\frac{\rho_{1}^{2}\left(  1-\left(
1-\rho_{1}\right)  ^{2\alpha}-2\alpha\rho_{1}\left(  1-\rho_{1}\right)
^{2\alpha-1}\right)  }{\left(  1-\left(  1-\rho_{1}\right)  ^{\alpha
+1}-\left(  \alpha+1\right)  \rho_{1}\left(  1-\rho_{1}\right)  ^{\alpha
}\right)  ^{2}}. \label{s}%
\end{equation}
The three results $\left(  \ref{Slim}\right)  $ allow us to construct an
estimator for the second-order parameter $\rho_{1}.$ Indeed, by recalling that
$n=n_{N}$ is a random sequence of integers, let $\upsilon=\upsilon_{n}$ be a
subsequence of $n,$ different than $k,$ such that given $n=m,$ $\upsilon
_{m}\rightarrow\infty,$ $\upsilon_{m}/m\rightarrow0$ as $N\rightarrow\infty.$
The sequence $\upsilon$ has to be chosen so that $\sqrt{\upsilon_{m}%
}\left\vert \mathbf{A}_{0}\left(  m/\upsilon_{m}\right)  \right\vert
\rightarrow\infty,$ which is a necessary condition to ensure the consistency
of $\rho_{1}$ estimator. On the other hand, as already pointed out, the
asymptotic normality of$\ \widehat{\gamma}_{1}^{\left(  BMN\right)  }$
requires that, for a given $n=m,$ $\sqrt{k_{m}}\mathbf{A}_{0}\left(
m/k_{m}\right)  \rightarrow\lambda<\infty.$ This means that both sample
fractions $k$ and $\upsilon$ have to be distinctly chosen. Since $\overline
{F}^{\ast}$ is regularly varying at infinity with index $-1/\gamma,$ then from
Lemma 3.2.1 in \cite{deHF06} page 69, we infer that, given $n=m,$ we have
$X_{m-\upsilon_{m}:m}\rightarrow\infty$ as $N\rightarrow\infty$ almost surely.
Then by using the total probability formula, we show that $X_{n-\upsilon
:n}\rightarrow\infty,$ almost surely too. By letting, in $\left(
\ref{Mt}\right)  ,$ $t=n/\upsilon$ then by replacing $\mathbb{U}_{F^{\ast}%
}\left(  n/\upsilon\right)  $ by $X_{n-\upsilon:n}$ and$\ \mathbf{F}$ by the
product-limit estimator $\mathbf{F}_{n},$ we get an estimator $M_{n}^{\left(
\alpha\right)  }\left(  \upsilon\right)  =M^{\left(  \alpha\right)  }\left(
t;\mathbf{F}_{n}\right)  $ for $M^{\left(  \alpha\right)  }\left(
t;\mathbf{F}\right)  $ as follows:%
\begin{equation}
M_{n}^{\left(  \alpha\right)  }\left(  \upsilon\right)  =\frac{1}%
{\overline{\mathbf{F}}_{n}\left(  X_{n-\upsilon:n}\right)  }\int
_{X_{n-\upsilon:n}}^{\infty}\log^{\alpha}\left(  x/X_{n-\upsilon:n}\right)
d\mathbf{F}_{n}\left(  x\right)  . \label{Mn}%
\end{equation}
Next, we give an explicit formula for $M_{n}^{\left(  \alpha\right)  }\left(
\upsilon\right)  $ in terms of observed sample $X_{1},...,X_{n}.$ Since
$\overline{\mathbf{F}}$ and $\overline{\mathbf{G}}$ are regularly varying with
negative indices $-1/\gamma_{1}$ and $-1/\gamma_{2}$ respectively, then their
right endpoints are infinite and thus they are equal. Hence, from \cite{W-85},
we may write $\int_{x}^{\infty}d\mathbf{F}\left(  y\right)  /\mathbf{F}\left(
y\right)  =\int_{x}^{\infty}dF^{\ast}\left(  y\right)  /C\left(  y\right)  ,$
where $C\left(  z\right)  :=\mathbf{P}\left(  X\leq z\leq Y\right)  $ is the
theoretical counterpart of $C_{n}\left(  z\right)  $ given in $(\ref{BMN}).$
Differentiating the previous two integrals leads to the crucial equation
$C\left(  x\right)  d\mathbf{F}\left(  x\right)  =\mathbf{F}\left(  x\right)
dF^{\ast}\left(  x\right)  ,$ which implies that $C_{n}\left(  x\right)
d\mathbf{F}_{n}\left(  x\right)  =\mathbf{F}_{n}\left(  x\right)  dF_{n}%
^{\ast}\left(  x\right)  ,$ where $F_{n}^{\ast}\left(  x\right)  :=n^{-1}%
\sum_{i=1}^{n}1\left(  X_{i}\leq x\right)  $ is the usual empirical df based
on the observed sample $X_{1},...,X_{n}.$ It follows that
\[
M_{n}^{\left(  \alpha\right)  }\left(  \upsilon\right)  =\frac{1}%
{\overline{\mathbf{F}}_{n}\left(  X_{n-\upsilon:n}\right)  }\int
_{X_{n-\upsilon:n}}^{\infty}\frac{\mathbf{F}_{n}\left(  x\right)  }%
{C_{n}\left(  x\right)  }\log^{\alpha}xdF_{n}^{\ast}\left(  x\right)  ,
\]
which equals%
\[
M_{n}^{\left(  \alpha\right)  }\left(  \upsilon\right)  =\frac{1}%
{n\overline{\mathbf{F}}_{n}\left(  X_{n-\upsilon:n}\right)  }\sum
_{i=1}^{\upsilon}\frac{\mathbf{F}_{n}\left(  X_{n-i+1:n}\right)  }%
{C_{n}\left(  X_{n-i+1:n}\right)  }\log^{\alpha}\frac{X_{n-i+1:n}%
}{X_{n-\upsilon:n}}.
\]
Similarly, we show that $\overline{\mathbf{F}}_{n}\left(  X_{n-\upsilon
:n}\right)  =n^{-1}\sum_{i=1}^{n}\mathbf{F}_{n}\left(  X_{n-i+1:n}\right)
/C_{n}\left(  X_{n-i+1:n}\right)  .$ This leads to following form of
$M^{\left(  \alpha\right)  }\left(  t;\mathbf{F}\right)  $ estimator:%
\[
M_{n}^{\left(  \alpha\right)  }\left(  \upsilon\right)  :=\sum_{i=1}%
^{\upsilon}a_{n}^{\left(  i\right)  }\log^{\alpha}\frac{X_{n-i+1:n}%
}{X_{n-\upsilon:n}}.
\]
It is readily observable that $M_{n}^{\left(  1\right)  }\left(  k\right)
=\widehat{\gamma}_{1}^{\left(  BMN\right)  }.$ Making use of $\left(
\ref{S}\right)  $ with the expression above, we get an estimator of
$S^{\left(  \alpha\right)  }\left(  t;\mathbf{F}\right)  ,$ that we denote
$S_{n}^{\left(  \alpha\right)  }=S_{n}^{\left(  \alpha\right)  }\left(
\upsilon\right)  .$ This, in virtue of the third limit in $\left(
\ref{Slim}\right)  ,$ leads to estimating $s_{\alpha}\left(  \rho_{1}\right)
.$ It is noteworthy that the function $\rho_{1}\rightarrow s_{\alpha}\left(
\rho_{1}\right)  ,$ defined and continuous on the set of negative real
numbers, is increasing for\textbf{\ }$0<\alpha<1/2$ and decreasing
for\textbf{\ }$\alpha>1/2,$\textbf{\ }$\alpha\neq1.$ Then, for suitable values
of $\alpha,$ we may invert $s_{\alpha}$ to get an estimator $\widehat{\rho
}_{1}^{\left(  \alpha\right)  }$ for $\rho_{1}$ as follows:%
\begin{equation}
\widehat{\rho}_{1}^{\left(  \alpha\right)  }:=s_{\alpha}^{\leftarrow}\left(
S_{n}^{\left(  \alpha\right)  }\right)  ,\text{ provided that }S_{n}^{\left(
\alpha\right)  }\in\mathcal{A}_{\alpha}, \label{rho1}%
\end{equation}
where $\mathcal{A}_{\alpha}$ is one of the following two regions:%
\[
\left\{  s:\left(  2\alpha-1\right)  /\alpha^{2}<s\leq4\left(  2\alpha
-1\right)  /\left(  \alpha\left(  \alpha+1\right)  \right)  ^{2},\text{ for
}\alpha\in\left(  0,1/2\right)  \right\}  ,
\]
or%
\[
\left\{  s:4\left(  2\alpha-1\right)  /\left(  \alpha\left(  \alpha+1\right)
\right)  ^{2}\leq s<\left(  2\alpha-1\right)  /\alpha^{2},\text{ for }%
\alpha\in\left(  1/2,\infty\right)  \backslash\left\{  1\right\}  \right\}  .
\]
For more details, regarding the construction of these two sets, one refers to
Remark 2.1 and Lemma 3.1 in \cite{GHP2003}. It is worth mentioning that, for
$\alpha=2,$ we have $s_{2}\left(  \rho_{1}\right)  =\left(  3\rho_{1}%
^{2}-8\rho_{1}+6\right)  /\left(  3-2\rho_{1}\right)  ^{2}$ and $s_{2}%
^{\leftarrow}\left(  s\right)  =\left(  6s-4+\sqrt{3s-2}\right)  /\left(
4s-3\right)  ,$ for $2/3<s<3/4.$ Thereby, we obtain an explicit formula to the
estimator of $\rho_{1}$ as follows%
\begin{equation}
\widehat{\rho}_{1}^{\left(  2\right)  }=\frac{6S_{n}^{\left(  2\right)
}-4+\sqrt{3S_{n}^{\left(  2\right)  }-2}}{4S_{n}^{\left(  2\right)  }%
-3},\text{ provided that }2/3<S_{n}^{\left(  2\right)  }<3/4. \label{rho1bis}%
\end{equation}
Next, we derive an asymptotically unbiased estimator for $\gamma_{1},$ that
improves $\widehat{\gamma}_{1}^{\left(  BMN\right)  }$ by estimating the
asymptotic bias $\mathbf{A}_{0}\left(  n/k\right)  /\left(  1-\rho_{1}\right)
,$ given in weak approximation $(\ref{approx}).$ Indeed, let $\upsilon$ be
equal to $u_{n}:=\left[  n^{1-\epsilon}\right]  ,$ for a fixed $\epsilon>0$
close to zero (say $\epsilon=0.01)$ so that, given $n=m,$ $u_{m}%
\rightarrow\infty,$ $u_{m}/m\rightarrow\infty$ and $\sqrt{u_{m}}\left\vert
\mathbf{A}_{0}\left(  m/u_{m}\right)  \right\vert \rightarrow\infty.$ The
validity of such a sequence is discussed in \cite{GHP2003} (Subsection 6.1,
conclusions 2 and 5). The estimator of $\rho_{1}$ pertaining to this choice of
$\upsilon$ will be denoted by $\widehat{\rho}_{1}^{\left(  \ast\right)  }.$ We
are now is proposition to define an estimator for $\mathbf{A}_{0}\left(
n/k\right)  .$ From assertion $(i)$ in Lemma $\ref{lemma1},$ we have
$\mathbf{A}_{0}\left(  t\right)  \sim\left(  1-\rho_{1}\right)  ^{2}\left(
M^{\left(  2\right)  }\left(  t;\mathbf{F}\right)  -2\left(  M^{\left(
1\right)  }\left(  t;\mathbf{F}\right)  \right)  ^{2}\right)  /\left(
2\rho_{1}M^{\left(  1\right)  }\left(  t;\mathbf{F}\right)  \right)  ,$ as
$t\rightarrow\infty.$ Then, by letting $t=n/k$ and by replacing, in the
previous quantity, $\mathbb{U}_{F^{\ast}}\left(  n/k\right)  $ by $X_{n-k:n},$
$\mathbf{F}$ by $\mathbf{F}_{n}$ and $\rho_{1}$ by $\widehat{\rho}%
_{1}^{\left(  \ast\right)  },$ we end up with
\[
\widehat{\mathbf{A}}_{0}\left(  n/k\right)  :=\left(  1-\widehat{\rho}%
_{1}^{\left(  \ast\right)  }\right)  ^{2}\left(  M_{n}^{\left(  2\right)
}\left(  k\right)  -2\left(  M_{n}^{\left(  1\right)  }\left(  k\right)
\right)  ^{2}\right)  /\left(  2\widehat{\rho}_{1}^{\left(  \ast\right)
}M_{n}^{\left(  1\right)  }\left(  k\right)  \right)  ,
\]
as an estimator for $\mathbf{A}_{0}\left(  n/k\right)  .$ Thus, we obtain an
asymptotically unbiased estimator%
\[
\widehat{\gamma}_{1}:=M_{n}^{\left(  1\right)  }\left(  k\right)  +\frac
{M_{n}^{\left(  2\right)  }\left(  k\right)  -2\left(  M_{n}^{\left(
1\right)  }\left(  k\right)  \right)  ^{2}}{2M_{n}^{\left(  1\right)  }\left(
k\right)  }\left(  1-\frac{1}{\widehat{\rho}_{1}^{\left(  \ast\right)  }%
}\right)  ,
\]
for the tail index $\gamma_{1},$ as an adaptation of Peng's estimator
\citep{peng98} to the random right-truncation case. The rest of the paper is
organized as follows. In Section \ref{sec2}, we present our main results which
consist in the consistency and the asymptotic normality of the estimators
$\widehat{\rho}_{1}^{\left(  \alpha\right)  }$ and $\widehat{\gamma}_{1}$
whose finite sample behaviours are checked by simulation in Section
\ref{sec3}. All proofs are gathered in Section \ref{sec4}. Two instrumental
lemmas are given in the Appendix.

\section{\textbf{Main results\label{sec2}}}

\noindent It is well known that, weak approximations of the second-order
parameter estimators are achieved in the third-order condition of regular
variation framework \citep[see, e.g.,
][]{deHS96}. Thus, it seems quite natural to suppose that df $\mathbf{F}%
$\ satisfies%
\begin{equation}
\lim_{t\rightarrow\infty}\left\{  \dfrac{\mathbb{U}_{\mathbf{F}}\left(
tx\right)  /\mathbb{U}_{\mathbf{F}}\left(  t\right)  -x^{\gamma_{1}}%
}{\mathbf{A}_{\mathbf{F}}\left(  t\right)  }-x^{\gamma_{1}}\dfrac{x^{\rho_{1}%
}-1}{\rho_{1}}\right\}  /\mathbf{B}_{\mathbf{F}}\left(  t\right)
=\frac{x^{\gamma_{1}}}{\rho_{1}}\left(  \dfrac{x^{\rho_{1}+\beta_{1}}-1}%
{\rho_{1}+\beta_{1}}-\dfrac{x^{\rho_{1}}-1}{\rho_{1}}\right)  , \label{Third}%
\end{equation}
where $\beta_{1}<0$ is the third-order parameter and $\mathbf{B}_{\mathbf{F}}%
$\ is a function tending to zero and not changing sign near infinity with
regularly varying absolute value at infinity with index $\beta_{1}.$ For
convenience, we set $\mathbf{B}_{0}\left(  t\right)  :=\mathbf{B}_{\mathbf{F}%
}\left(  1/\overline{\mathbf{F}}\left(  \mathbb{U}_{F^{\ast}}\left(  t\right)
\right)  \right)  $ and by keeping similar notations to those used in
\cite{GHP2003}, we write
\begin{equation}
\mu_{\alpha}^{\left(  1\right)  }:=\Gamma\left(  \alpha+1\right)  ,\text{ }%
\mu_{\alpha}^{\left(  2\right)  }\left(  \rho_{1}\right)  :=\frac
{\Gamma\left(  \alpha\right)  \left(  1-\left(  1-\rho_{1}\right)  ^{\alpha
}\right)  }{\rho_{1}\left(  1-\rho_{1}\right)  ^{\alpha}}, \label{mu1-2}%
\end{equation}%
\[
\mu_{\alpha}^{\left(  3\right)  }\left(  \rho_{1}\right)  :=\left\{
\begin{tabular}
[c]{ll}%
$\dfrac{1}{\rho_{1}^{2}}\log\dfrac{\left(  1-\rho_{1}\right)  ^{2}}%
{1-2\rho_{1}}$ & if $\alpha=1,$\\
$\dfrac{\Gamma\left(  \alpha\right)  }{\rho_{1}^{2}\left(  \alpha-1\right)
}\left\{  \dfrac{1}{\left(  1-2\rho_{1}\right)  ^{\alpha-1}}-\dfrac{2}{\left(
1-\rho_{1}\right)  ^{\alpha-1}}+1\right\}  $ & if $\alpha\neq1,$%
\end{tabular}
\ \ \ \ \ \ \right.  \medskip
\]%
\[
\mu_{\alpha}^{\left(  4\right)  }\left(  \rho_{1},\beta_{1}\right)
:=\beta_{1}^{-1}\left(  \mu_{\alpha}^{\left(  2\right)  }\left(  \rho
_{1}+\beta_{1}\right)  -\mu_{\alpha}^{\left(  2\right)  }\left(  \rho
_{1}\right)  \right)  ,
\]%
\[
m_{\alpha}:=\mu_{\alpha}^{\left(  2\right)  }\left(  \rho_{1}\right)
-\mu_{\alpha}^{\left(  1\right)  }\mu_{1}^{\left(  2\right)  }\left(  \rho
_{1}\right)  ,\text{ }c_{\alpha}:=\mu_{\alpha}^{\left(  3\right)  }\left(
\rho_{1}\right)  -\mu_{\alpha}^{\left(  1\right)  }\left(  \mu_{1}^{\left(
2\right)  }\left(  \rho_{1}\right)  \right)  ^{2}%
\]
and $d_{\alpha}:=\mu_{\alpha}^{\left(  4\right)  }\left(  \rho_{1},\beta
_{1}\right)  -\mu_{\alpha}^{\left(  1\right)  }\mu_{1}^{\left(  4\right)
}\left(  \varrho_{1},\beta_{1}\right)  .$ For further use, we set $r_{\alpha
}:=2q_{\alpha}\gamma_{1}^{2-\alpha}\Gamma\left(  \alpha+1\right)  ,$
\[
\eta_{1}:=\frac{1}{2\gamma_{1}m_{2}r_{\alpha+1}^{2}}\left(  \frac{\left(
2\alpha-1\right)  c_{2\alpha}}{\Gamma\left(  2\alpha\right)  }+c_{2}%
r_{2\alpha}-\frac{2c_{\alpha+1}r_{2\alpha}}{r_{\alpha+1}\Gamma\left(
\alpha\right)  }\right)  ,
\]%
\[
\eta_{2}:=\frac{1}{\gamma_{1}m_{2}r_{\alpha+1}^{2}}\frac{d_{2\alpha}}%
{\Gamma\left(  2\alpha\right)  }+d_{2}r_{2\alpha}-\frac{2d_{\alpha
+1}r_{2\alpha}}{r_{\alpha+1}\Gamma\left(  \alpha+1\right)  },
\]%
\[
\xi:=\gamma\left(  \frac{1-2\alpha-3r_{2\alpha}}{r_{\alpha+1}^{2}m_{2}}%
+\frac{2\alpha r_{2\alpha}}{r_{\alpha+1}^{3}m_{2}}\right)  ,
\]%
\[
\tau_{1}:=\frac{1}{\gamma_{1}^{2\alpha-1}r_{\alpha+1}^{2}\Gamma\left(
2\alpha+1\right)  m_{2}},\text{ }\tau_{2}:=-\frac{2r_{2\alpha}}{\gamma
_{1}^{\alpha}r_{\alpha+1}^{3}\Gamma\left(  \alpha+2\right)  m_{2}},
\]%
\[
\tau_{3}:=\frac{r_{2\alpha}}{\gamma_{1}r_{\alpha+1}^{2}2m_{2}},\text{ }%
\tau_{4}:=\frac{-2\alpha r_{\alpha+1}+2\left(  \alpha+1\right)  r_{2\alpha
}-4r_{\alpha+1}r_{2\alpha}}{r_{\alpha+1}^{3}m_{2}},
\]%
\[
\tau_{5}:=\frac{\rho_{1}-1}{2\gamma_{1}\rho_{1}},\text{ }\tau_{6}%
:=1+2\frac{1-\rho_{1}}{\gamma_{1}\rho_{1}}\text{ and }\mu:=\gamma\left(
2+2\frac{1-\rho_{1}}{\gamma_{1}\rho_{1}}-\frac{1}{\rho_{1}}\right)  .
\]

\begin{theorem}
\label{Theorem1}Assume that both df's $\mathbf{F}$ and $\mathbf{G}$ satisfy
the second-order conditions $(\ref{second-order})$ and $(\ref{second-orderG})$
respectively with $\gamma_{1}<\gamma_{2}.$ Let $\alpha,$ defined in
$(\ref{rho1}),$ be fixed and let $\upsilon$ be a random sequence of integers
such that, given $n=m,$ $\upsilon_{m}\rightarrow\infty,$ $\upsilon
_{m}/m\rightarrow0$ and $\sqrt{\upsilon_{m}}\left\vert \mathbf{A}_{0}\left(
m/\upsilon_{m}\right)  \right\vert \rightarrow\infty,$ then
\[
\widehat{\rho}_{1}^{\left(  \alpha\right)  }\overset{\mathbf{P}}{\rightarrow
}\rho_{1},\text{ as }N\rightarrow\infty.
\]
If in addition, we assume that the third-order condition $(\ref{Third})$
holds, then whenever, given $n=m,$ $\sqrt{\upsilon_{m}}\mathbf{A}_{0}%
^{2}\left(  m/\upsilon_{m}\right)  $ and $\sqrt{\upsilon_{m}}\mathbf{A}%
_{0}\left(  m/\upsilon_{m}\right)  \mathbf{B}_{0}\left(  m/\upsilon
_{m}\right)  $ are asymptotically bounded, then there exists a standard Wiener
process $\left\{  \mathbf{W}\left(  s\right)  ;\text{ }s\geq0\right\}
,$\ defined on the probability space $\left(  \Omega,\mathcal{A}%
,\mathbf{P}\right)  ,$ such that%
\begin{align*}
s_{\alpha}^{\prime}\left(  \rho_{1}\right)  \sqrt{\upsilon}\mathbf{A}%
_{0}\left(  n/\upsilon\right)  \left(  \widehat{\rho}_{1}^{\left(
\alpha\right)  }-\rho_{1}\right)   &  =\int_{0}^{1}s^{-\gamma/\gamma_{2}%
-1}\Delta_{\alpha}(s)\mathbf{W}\left(  s\right)  ds-\xi\mathbf{W}\left(
1\right) \\
&  +\eta_{1}\sqrt{\upsilon}\mathbf{A}_{0}^{2}\left(  n/\upsilon\right)
+\eta_{2}\sqrt{\upsilon}\mathbf{A}_{0}\left(  n/\upsilon\right)
\mathbf{B}_{0}\left(  n/\upsilon\right)  +o_{\mathbf{P}}\left(  1\right)  ,
\end{align*}
where $s_{\alpha}^{\prime}$ is the Lebesgue derivative of $s_{\alpha}$ given
in $(\ref{s})$ and%
\begin{align*}
&  \Delta_{\alpha}(s)%
\begin{tabular}
[c]{l}%
$:=$%
\end{tabular}
\ \ \ \ \frac{\tau_{1}\gamma\log^{2\alpha}s^{-\gamma}}{\gamma_{1}+\gamma_{2}%
}+\frac{2\alpha\tau_{1}\gamma^{2}\log^{2\alpha-1}s^{-\gamma}}{\gamma_{1}%
}+\frac{\tau_{2}\gamma\log^{\alpha+1}s^{-\gamma}}{\gamma_{1}+\gamma_{2}}\\
\ \ \ \ \  &  \ \ \ \ \ \ \ \ \ \ \ \ \ +\frac{\tau_{2}\left(  \alpha
+1\right)  \gamma^{2}\log^{\alpha}s^{-\gamma}}{\gamma_{1}}+\frac{\tau
_{3}\gamma\log^{2}s^{-\gamma}}{\gamma_{1}+\gamma_{2}}\\
&  \ \ \ \ \ \ \ \ \ \ \ \ \ +\left(  \frac{2\tau_{3}\gamma^{2}}{\gamma_{1}%
}+\frac{\tau_{4}\gamma}{\gamma_{1}+\gamma_{2}}\right)  \log s^{-\gamma}%
+\frac{\tau_{4}\gamma^{2}}{\gamma_{1}}-\frac{\gamma_{1}\xi}{\gamma_{1}%
+\gamma_{2}}%
\end{align*}
If, in addition, we suppose that given $n=m,$
\[
\sqrt{\upsilon_{m}}\mathbf{A}_{0}^{2}\left(  m/\upsilon_{m}\right)
\rightarrow\lambda_{1}<\infty\text{ and }\sqrt{\upsilon_{m}}\mathbf{A}%
_{0}\left(  m/\upsilon_{m}\right)  \mathbf{B}_{0}\left(  m/\upsilon
_{m}\right)  \rightarrow\lambda_{2}<\infty,
\]
then $\sqrt{\upsilon}\mathbf{A}_{0}\left(  n/\upsilon\right)  \left(
\widehat{\rho}_{1}^{\left(  \alpha\right)  }-\rho_{1}\right)  \overset
{\mathcal{D}}{\rightarrow}\mathcal{N}\left(  \eta_{1}\lambda_{1}+\eta
_{2}\lambda_{2},\sigma_{\alpha}^{2}\right)  ,$ as $N\rightarrow\infty,$ where
\[
\sigma_{\alpha}^{2}:=%
%TCIMACRO{\dint _{0}^{1}}%
%BeginExpansion
{\displaystyle\int_{0}^{1}}
%EndExpansion%
%TCIMACRO{\dint _{0}^{1}}%
%BeginExpansion
{\displaystyle\int_{0}^{1}}
%EndExpansion
s^{-\gamma/\gamma_{2}-1}t^{-\gamma/\gamma_{2}-1}\min\left(  s,t\right)
\Delta_{\alpha}(s)\Delta_{\alpha}(t)dsdt-2\xi%
%TCIMACRO{\dint _{0}^{1}}%
%BeginExpansion
{\displaystyle\int_{0}^{1}}
%EndExpansion
s^{-\gamma/\gamma_{2}}\Delta_{\alpha}(s)ds+\xi^{2}.
\]

\end{theorem}

\begin{theorem}
\label{Theorem2} Let $k$ be a random sequence of integers, different from
$\upsilon,$ such that, given $n=m,$ $k_{m}\rightarrow\infty,$ $k_{m}%
/m\rightarrow0$ and $\sqrt{k_{m}}\mathbf{A}_{0}\left(  m/k_{m}\right)  $ is
asymptotically bounded, then with the same Wiener process $\left\{
\mathbf{W}\left(  s\right)  ;\text{ }s\geq0\right\}  ,$\ for any $\epsilon>0,$
we have%
\[
\sqrt{k}\left(  \widehat{\gamma}_{1}-\gamma_{1}\right)  =\int_{0}%
^{1}s^{-\gamma/\gamma_{2}-1}\mathbf{D}(s)\mathbf{W}\left(  s\right)
ds-\mu\mathbf{W}\left(  1\right)  +o_{\mathbf{P}}\left(  1\right)  ,
\]
where%
\[
\mathbf{D}(s):=\frac{\gamma^{3}\tau_{5}}{\gamma_{1}+\gamma_{2}}\log
^{2}s-\left(  \frac{2\tau_{5}\gamma^{3}}{\gamma_{1}}+\frac{\gamma^{2}\tau_{6}%
}{\gamma_{1}+\gamma_{2}}\right)  \log s+\frac{\tau_{6}\gamma^{2}}{\gamma_{1}%
}-\frac{\gamma_{1}\mu}{\gamma_{1}+\gamma_{2}}.
\]
If, in addition, we suppose that, given $n=m,$ $\sqrt{k_{m}}\mathbf{A}%
_{0}\left(  m/k_{m}\right)  \rightarrow\lambda<\infty,$ then
\[
\sqrt{k}\left(  \widehat{\gamma}_{1}-\gamma_{1}\right)  \overset{\mathcal{D}%
}{\rightarrow}\mathcal{N}\left(  0,\sigma_{\ast}^{2}\right)  ,\ \text{as
}N\rightarrow\infty,
\]
where $\sigma_{\ast}^{2}:=%
%TCIMACRO{\dint _{0}^{1}}%
%BeginExpansion
{\displaystyle\int_{0}^{1}}
%EndExpansion%
%TCIMACRO{\dint _{0}^{1}}%
%BeginExpansion
{\displaystyle\int_{0}^{1}}
%EndExpansion
s^{-\gamma/\gamma_{2}-1}t^{-\gamma/\gamma_{2}-1}\min\left(  s,t\right)
\mathbf{D}(s)\mathbf{D}(t)dsdt-2\mu%
%TCIMACRO{\dint _{0}^{1}}%
%BeginExpansion
{\displaystyle\int_{0}^{1}}
%EndExpansion
s^{-\gamma/\gamma_{2}}\mathbf{D}(s)ds+\mu^{2}.$
\end{theorem}

\section{\textbf{Simulation study}\label{sec3}}

\noindent In this section, we study the performance of $\widehat{\rho}%
_{1}^{\left(  \alpha\right)  }$ (for $\alpha=2)$ and compare the newly
introduced bias-reduced estimator $\widehat{\gamma}_{1}$ with $\widehat
{\gamma}_{1}^{\left(  BMN\right)  }.$ Let us consider two sets of truncated
and truncation data respectively drawn from Burr's models, $\overline
{\mathbf{F}}\left(  x\right)  =\left(  1+x^{1/\delta}\right)  ^{-\delta
/\gamma_{1}}$ and $\overline{\mathbf{G}}\left(  x\right)  =\left(
1+x^{1/\delta}\right)  ^{-\delta/\gamma_{2}},$ $x\geq0,$ where $\delta
,\gamma_{1},\gamma_{2}>0.$ By elementary analysis, it is easy to verify that
$\overline{\mathbf{F}}$ satisfies the third-order condition $(\ref{Third})$
with $\rho_{1}=\beta_{1}=-\gamma_{1}/\delta,$ $\mathbf{A}_{\mathbf{F}}\left(
x\right)  =\gamma_{1}x^{\rho_{1}}/\left(  1-x^{\rho_{1}}\right)  $ and
$\mathbf{B}_{\mathbf{F}}\left(  x\right)  =\left(  \delta/\gamma_{1}\right)
\mathbf{A}_{\mathbf{F}}\left(  x\right)  .$ We fix $\delta=1/4$ and choose the
values $0.6$ and $0.8$ for $\gamma_{1}$ and $70\%$ and $90\%$ for the
percentage of observed data $p=\gamma_{2}/(\gamma_{1}+\gamma_{2}).$ For each
couple $\left(  \gamma_{1},p\right)  ,$ we solve the latter equation to get
the pertaining $\gamma_{2}$-value. We vary the common size $N$ of both samples
$\left(  \mathbf{X}_{1},...,\mathbf{X}_{N}\right)  $ and $\left(
\mathbf{Y}_{1},...,\mathbf{Y}_{N}\right)  ,$ then for each size, we generate
$1000$ independent replicates. For the selection of the optimal numbers
$\upsilon^{\ast}$ and $k^{\ast}$ of upper order statistics used in the
computation of estimators $\widehat{\rho}_{1}^{\left(  2\right)  },$
$\widehat{\gamma}_{1}$ and $\widehat{\gamma}_{1}^{\left(  BMN\right)  },$ we
apply the algorithm of \cite{ReTo7}, page 137. Our illustration and
comparison, made with respect to the absolute biases (abias) and rmse's, are
summarized in Tables \ref{Tab1} and \ref{Tab2}. The obtained results, in Table
\ref{Tab1}, show that $\widehat{\rho}_{1}^{\left(  2\right)  }$ behaves well
in terms of bias and rmse and it is clear that from Table \ref{Tab2} that
$\widehat{\gamma}_{1}$ performs better $\widehat{\gamma}_{1}^{\left(
BMN\right)  }$ both in bias and remse too.%

%TCIMACRO{\TeXButton{B}{\begin{table}[tbp] \centering}}%
%BeginExpansion
\begin{table}[tbp] \centering
%EndExpansion%
\begin{tabular}
[c]{ccccccccc}
& \multicolumn{4}{c}{$p=0.7$} & \multicolumn{4}{c}{$p=0.9$}\\\hline
$N$ & \multicolumn{1}{|c}{$n$} & $\upsilon^{\ast}$ & abias & rmse &
\multicolumn{1}{|c}{$n$} & $\upsilon^{\ast}$ & abias & rmse\\\hline\hline
\multicolumn{9}{c}{$\gamma_{1}=0.6$}\\\hline\hline
\multicolumn{1}{l}{$100$} & \multicolumn{1}{|l}{$70$} &
\multicolumn{1}{l}{$27$} & \multicolumn{1}{l}{$0.009$} &
\multicolumn{1}{l}{$0.047$} & \multicolumn{1}{|l}{$89$} &
\multicolumn{1}{l}{$38$} & \multicolumn{1}{l}{$0.004$} &
\multicolumn{1}{l}{$0.048$}\\
\multicolumn{1}{l}{$200$} & \multicolumn{1}{|c}{$151$} & $70$ & $0.008$ &
\multicolumn{1}{l}{$0.046$} & \multicolumn{1}{|c}{$179$} & $73$ &
\multicolumn{1}{l}{$0.003$} & \multicolumn{1}{l}{$0.046$}\\
\multicolumn{1}{l}{$500$} & \multicolumn{1}{|c}{$349$} & $208$ & $0.005$ &
\multicolumn{1}{l}{$0.043$} & \multicolumn{1}{|c}{$450$} & $243$ &
\multicolumn{1}{l}{$0.002$} & \multicolumn{1}{l}{$0.048$}\\
\multicolumn{1}{l}{$1000$} & \multicolumn{1}{|c}{$697$} & $667$ & $0.001$ &
\multicolumn{1}{l}{$0.027$} & \multicolumn{1}{|c}{$896$} & $641$ &
\multicolumn{1}{l}{$0.001$} & \multicolumn{1}{l}{$0.030$}\\\hline\hline
\multicolumn{9}{c}{$\gamma_{1}=0.8$}\\\hline\hline
\multicolumn{1}{l}{$100$} & \multicolumn{1}{|l}{$70$} &
\multicolumn{1}{l}{$30$} & \multicolumn{1}{l}{$0.011$} &
\multicolumn{1}{l}{$0.050$} & \multicolumn{1}{|l}{$90$} &
\multicolumn{1}{l}{$40$} & \multicolumn{1}{l}{$0.013$} &
\multicolumn{1}{l}{$0.048$}\\
\multicolumn{1}{l}{$200$} & \multicolumn{1}{|c}{$139$} & $67$ & $0.009$ &
\multicolumn{1}{l}{$0.048$} & \multicolumn{1}{|c}{$179$} & $71$ &
\multicolumn{1}{l}{$0.012$} & \multicolumn{1}{l}{$0.047$}\\
\multicolumn{1}{l}{$500$} & \multicolumn{1}{|c}{$350$} & $198$ & $0.008$ &
\multicolumn{1}{l}{$0.043$} & \multicolumn{1}{|c}{$449$} & $232$ &
\multicolumn{1}{l}{$0.006$} & \multicolumn{1}{l}{$0.049$}\\
\multicolumn{1}{l}{$1000$} & \multicolumn{1}{|c}{$730$} & $301$ & $0.002$ &
\multicolumn{1}{l}{$0.037$} & \multicolumn{1}{|c}{$903$} & $378$ &
\multicolumn{1}{l}{$0.002$} & \multicolumn{1}{l}{$0.029$}\\\hline\hline
&  &  &  &  &  &  &  &
\end{tabular}
\caption{Absolute bias and rmse of the second-order parameter estimator
based on 1000 right-truncated samples of Burr's models.}\label{Tab1}%
%TCIMACRO{\TeXButton{E}{\end{table}}}%
%BeginExpansion
\end{table}%
%EndExpansion
%

%TCIMACRO{\TeXButton{B}{\begin{table}[tbp] \centering}}%
%BeginExpansion
\begin{table}[tbp] \centering
%EndExpansion%
\begin{tabular}
[c]{ccccccccccccccc}
& \multicolumn{7}{c}{$p=0.7$} & \multicolumn{7}{||c}{$p=0.9$}\\\hline
&  & \multicolumn{3}{c}{$\widehat{\gamma}_{1}$} &
\multicolumn{3}{|c}{$\widehat{\gamma}_{1}^{\left(  BMN\right)  }$} &
\multicolumn{1}{||c}{} & \multicolumn{3}{c}{$\widehat{\gamma}_{1}$} &
\multicolumn{3}{|c}{$\widehat{\gamma}_{1}^{\left(  BMN\right)  }$}\\\hline
$N$ & \multicolumn{1}{|c}{$n$} & $k^{\ast}$ & abias & rmse &
\multicolumn{1}{|c}{$k^{\ast}$} & abias & rmse & \multicolumn{1}{||c}{$n$} &
$k^{\ast}$ & abias & rmse & \multicolumn{1}{|c}{$k^{\ast}$} & abias &
rmse\\\hline\hline
\multicolumn{15}{c}{$\gamma_{1}=0.6$}\\\hline\hline
\multicolumn{1}{l}{${\small 100}$} & \multicolumn{1}{|l}{${\small 70}$} &
\multicolumn{1}{l}{${\small 12}$} & \multicolumn{1}{l}{${\small 0.068}$} &
\multicolumn{1}{l}{${\small 0.263}$} & \multicolumn{1}{|c}{${\small 11}$} &
${\small 0.127}$ & ${\small 0.259}$ & \multicolumn{1}{||l}{${\small 89}$} &
\multicolumn{1}{l}{${\small 16}$} & \multicolumn{1}{l}{${\small 0.013}$} &
\multicolumn{1}{l}{${\small 0.152}$} & \multicolumn{1}{|c}{${\small 15}$} &
${\small 0.118}$ & ${\small 0.217}$\\
\multicolumn{1}{l}{${\small 200}$} & \multicolumn{1}{|c}{${\small 140}$} &
${\small 26}$ & ${\small 0.048}$ & \multicolumn{1}{l}{${\small 0.200}$} &
\multicolumn{1}{|c}{${\small 24}$} & ${\small 0.090}$ & ${\small 0.223}$ &
\multicolumn{1}{||c}{${\small 180}$} & ${\small 34}$ & ${\small 0.006}$ &
\multicolumn{1}{l}{${\small 0.116}$} & \multicolumn{1}{|c}{${\small 31}$} &
${\small 0.089}$ & ${\small 0.176}$\\
\multicolumn{1}{l}{${\small 500}$} & \multicolumn{1}{|c}{${\small 349}$} &
${\small 63}$ & ${\small 0.020}$ & \multicolumn{1}{l}{${\small 0.123}$} &
\multicolumn{1}{|c}{${\small 58}$} & ${\small 0.072}$ & ${\small 0.173}$ &
\multicolumn{1}{||c}{${\small 449}$} & ${\small 83}$ & ${\small 0.002}$ &
\multicolumn{1}{l}{${\small 0.078}$} & \multicolumn{1}{|c}{${\small 78}$} &
${\small 0.056}$ & ${\small 0.129}$\\
\multicolumn{1}{l}{${\small 1000}$} & \multicolumn{1}{|c}{${\small 703}$} &
${\small 115}$ & ${\small 0.007}$ & \multicolumn{1}{l}{${\small 0.097}$} &
\multicolumn{1}{|c}{${\small 112}$} & ${\small 0.011}$ & ${\small 0.121}$ &
\multicolumn{1}{||c}{${\small 898}$} & ${\small 176}$ & ${\small 0.001}$ &
\multicolumn{1}{l}{${\small 0.037}$} & \multicolumn{1}{|c}{${\small 174}$} &
${\small 0.016}$ & ${\small 0.056}$\\\hline\hline
\multicolumn{15}{c}{$\gamma_{1}=0.8$}\\\hline\hline
\multicolumn{1}{l}{${\small 100}$} & \multicolumn{1}{|l}{${\small 70}$} &
\multicolumn{1}{l}{${\small 13}$} & \multicolumn{1}{l}{${\small 0.067}$} &
\multicolumn{1}{l}{${\small 0.311}$} & \multicolumn{1}{|c}{${\small 12}$} &
${\small 0.222}$ & ${\small 0.217}$ & \multicolumn{1}{||l}{${\small 89}$} &
\multicolumn{1}{l}{${\small 16}$} & \multicolumn{1}{l}{${\small 0.063}$} &
\multicolumn{1}{l}{${\small 0.220}$} & \multicolumn{1}{|c}{${\small 15}$} &
${\small 0.196}$ & ${\small 0.315}$\\
\multicolumn{1}{l}{${\small 200}$} & \multicolumn{1}{|c}{${\small 140}$} &
${\small 25}$ & ${\small 0.014}$ & \multicolumn{1}{l}{${\small 0.219}$} &
\multicolumn{1}{|c}{${\small 24}$} & ${\small 0.163}$ & ${\small 0.282}$ &
\multicolumn{1}{||c}{${\small 179}$} & ${\small 33}$ & ${\small 0.033}$ &
\multicolumn{1}{l}{${\small 0.150}$} & \multicolumn{1}{|c}{${\small 31}$} &
${\small 0.131}$ & ${\small 0.220}$\\
\multicolumn{1}{l}{${\small 500}$} & \multicolumn{1}{|c}{${\small 349}$} &
${\small 64}$ & ${\small 0.011}$ & \multicolumn{1}{l}{${\small 0.152}$} &
\multicolumn{1}{|c}{${\small 59}$} & ${\small 0.033}$ & ${\small 0.222}$ &
\multicolumn{1}{||c}{${\small 449}$} & ${\small 85}$ & ${\small 0.021}$ &
\multicolumn{1}{l}{${\small 0.097}$} & \multicolumn{1}{|c}{${\small 79}$} &
${\small 0.088}$ & ${\small 0.156}$\\
\multicolumn{1}{l}{${\small 1000}$} & \multicolumn{1}{|c}{${\small 707}$} &
${\small 145}$ & ${\small 0.007}$ & \multicolumn{1}{l}{${\small 0.054}$} &
\multicolumn{1}{|c}{${\small 125}$} & ${\small 0.017}$ & ${\small 0.133}$ &
\multicolumn{1}{||c}{${\small 897}$} & ${\small 179}$ & ${\small 0.013}$ &
\multicolumn{1}{l}{${\small 0.058}$} & \multicolumn{1}{|c}{${\small 166}$} &
${\small 0.019}$ & ${\small 0.098}$\\\hline\hline
&  &  &  &  &  &  &  &  &  &  &  &  &  &
\end{tabular}
\caption{Absolute biases and rmse's of  the tail index estimators
based on 1000 right-truncated samples of Burr's models. }\label{Tab2}%
%TCIMACRO{\TeXButton{E}{\end{table}}}%
%BeginExpansion
\end{table}%
%EndExpansion

\section{\textbf{Proofs}\label{sec4}}

\subsection{\noindent\textbf{Proof of Theorem }$\ref{Theorem1}$}

\noindent We begin by proving the consistency of $\widehat{\rho}_{1}^{\left(
\alpha\right)  }$ defined in $(\ref{rho1}).$ We let%
\[
\mathbf{L}_{n}\left(  x;\upsilon\right)  :=\frac{\overline{\mathbf{F}}%
_{n}\left(  X_{n-\upsilon:n}x\right)  }{\overline{\mathbf{F}}_{n}\left(
X_{n-\upsilon:n}\right)  }-\frac{\overline{\mathbf{F}}\left(  X_{n-\upsilon
:n}x\right)  }{\overline{\mathbf{F}}\left(  X_{n-\upsilon:n}\right)  },
\]
and we show that for any $\alpha>0$
\begin{equation}
M_{n}^{\left(  \alpha\right)  }\left(  \upsilon\right)  =\gamma_{1}^{\alpha
}\mu_{\alpha}^{\left(  1\right)  }+\int_{1}^{\infty}\mathbf{L}_{n}\left(
x;\upsilon\right)  d\log^{\alpha}x+\left(  1+o_{\mathbf{P}}\left(  1\right)
\right)  \alpha\gamma_{1}^{\alpha-1}\mu_{\alpha}^{\left(  2\right)  }\left(
\rho_{1}\right)  \mathbf{A}_{0}\left(  n/\upsilon\right)  , \label{M-approx}%
\end{equation}
where $\mu_{\alpha}^{\left(  1\right)  }$ and $\mu_{\alpha}^{\left(  2\right)
}\left(  \rho_{1}\right)  $ are as in $(\ref{mu1-2}).$ It is clear that from
formula $\left(  \ref{Mn}\right)  ,$ $M_{n}^{\left(  \alpha\right)  }\left(
\upsilon\right)  $ may be rewritten into $-\int_{1}^{\infty}\log^{\alpha
}xd\overline{\mathbf{F}}_{n}\left(  X_{n-\upsilon:n}x\right)  /\overline
{\mathbf{F}}_{n}\left(  X_{n-\upsilon:n}\right)  ,$ which by an integration by
parts equals $\int_{1}^{\infty}\overline{\mathbf{F}}_{n}\left(  X_{n-\upsilon
:n}x\right)  /\overline{\mathbf{F}}_{n}\left(  X_{n-\upsilon:n}\right)
d\log^{\alpha}x.$ The latter may be decomposed into%
\[
\int_{1}^{\infty}\mathbf{L}_{n}\left(  x;\upsilon\right)  d\log^{\alpha}%
x+\int_{1}^{\infty}\left(  \frac{\overline{\mathbf{F}}_{n}\left(
X_{n-\upsilon:n}x\right)  }{\overline{\mathbf{F}}_{n}\left(  X_{n-\upsilon
:n}\right)  }-x^{-1/\gamma_{1}}\right)  d\log^{\alpha}x+\int_{1}^{\infty
}x^{-1/\gamma_{1}}d\log^{\alpha}x.
\]
It is easy to verify that $\int_{1}^{\infty}x^{-1/\gamma_{1}}d\log^{\alpha}x$
equals $\gamma_{1}^{\alpha}\mu_{\alpha}^{\left(  1\right)  }.$ Since,
$X_{n-\upsilon:n}\rightarrow\infty,$ almost surely, then by making use of the
uniform inequality of the second-order regularly varying functions, to
$\overline{\mathbf{F}},$ given in Proposition 4 of \cite{Hua}, we write: with
probability one, for any $0<\epsilon<1$ and large $N$%
\begin{equation}
\left\vert \frac{\overline{\mathbf{F}}\left(  X_{n-\upsilon:n}x\right)
/\overline{\mathbf{F}}\left(  X_{n-\upsilon:n}\right)  -x^{-1/\gamma_{1}}%
}{\gamma_{1}^{-2}\widetilde{\mathbf{A}}_{\mathbf{F}}\left(  1/\overline
{\mathbf{F}}\left(  X_{n-\upsilon:n}\right)  \right)  }-x^{-1/\gamma_{1}%
}\dfrac{x^{\rho_{1}/\gamma_{1}}-1}{\gamma_{1}/\rho_{1}}\right\vert
\leq\epsilon x^{-1/\gamma_{1}+\epsilon},\text{ for any }x\geq1,
\label{Potter1}%
\end{equation}
where $\widetilde{\mathbf{A}}_{\mathbf{F}}\left(  t\right)  \sim
\mathbf{A}_{\mathbf{F}}\left(  t\right)  ,$ as $t\rightarrow\infty.$ This
implies, almost surely, that
\begin{align*}
&  \int_{1}^{\infty}\left(  \frac{\overline{\mathbf{F}}_{n}\left(
X_{n-\upsilon:n}x\right)  }{\overline{\mathbf{F}}_{n}\left(  X_{n-\upsilon
:n}\right)  }-x^{-1/\gamma_{1}}\right)  d\log^{\alpha}x\\
&  =\widetilde{\mathbf{A}}_{\mathbf{F}}\left(  1/\overline{\mathbf{F}}\left(
X_{n-\upsilon:n}\right)  \right)  \left\{  \int_{1}^{\infty}x^{-1/\gamma_{1}%
}\dfrac{x^{\rho_{1}/\gamma_{1}}-1}{\gamma_{1}\rho_{1}}d\log^{\alpha
}x+o_{\mathbf{P}}\left(  \int_{1}^{\infty}x^{-1/\gamma_{1}+\epsilon}%
d\log^{\alpha}x\right)  \right\}  .
\end{align*}
We check that $\int_{1}^{\infty}x^{-1/\gamma_{1}}\dfrac{x^{\rho_{1}/\gamma
_{1}}-1}{\gamma_{1}\rho_{1}}d\log^{\alpha}x=\alpha\gamma_{1}^{\alpha-1}%
\mu_{\alpha}^{\left(  2\right)  }\left(  \rho_{1}\right)  $ and $\int
_{1}^{\infty}x^{-1/\gamma_{1}+\epsilon}d\log^{\alpha}x$ is finite. From Lemma
7.4 in \cite{BchMN-16a}, $X_{n-\upsilon:n}/\mathbb{U}_{F^{\ast}}\left(
n/\upsilon\right)  \overset{\mathbf{P}}{\rightarrow}1,$ as $N\rightarrow
\infty,$ then by using the regular variation property of $\left\vert
\mathbf{A}_{\mathbf{F}}\left(  1/\overline{\mathbf{F}}\left(  \cdot\right)
\right)  \right\vert $ and the corresponding Potter's inequalities (see, for
instance, Proposition B.1.10 in \cite{deHF06}), we get
\[
\widetilde{\mathbf{A}}_{\mathbf{F}}\left(  1/\overline{\mathbf{F}}\left(
X_{n-\upsilon:n}\right)  \right)  =\left(  1+o_{\mathbf{P}}\left(  1\right)
\right)  \mathbf{A}_{\mathbf{F}}\left(  1/\overline{\mathbf{F}}\left(
\mathbb{U}_{F^{\ast}}\left(  n/\upsilon\right)  \right)  \right)  =\left(
1+o_{\mathbf{P}}\left(  1\right)  \right)  \mathbf{A}_{0}\left(
n/\upsilon\right)  ,
\]
therefore
\[
M_{n}^{\left(  \alpha\right)  }\left(  \upsilon\right)  =\gamma_{1}^{\alpha
}\mu_{\alpha}^{\left(  1\right)  }+\int_{1}^{\infty}\mathbf{L}_{n}\left(
x;\upsilon\right)  d\log^{\alpha}x+\alpha\gamma_{1}^{\alpha-1}\mu_{\alpha
}^{\left(  2\right)  }\left(  \rho_{1}\right)  \mathbf{A}_{0}\left(
n/\upsilon\right)  \left(  1+o_{\mathbf{P}}\left(  1\right)  \right)  .
\]
In the second step, we use the Gaussian approximation of $\mathbf{L}%
_{n}\left(  x\right)  $ recently given by \cite{BchMN-16a} (assertion $\left(
6.26\right)  ),$ saying that: for any $0<\epsilon<1/2-\gamma/\gamma_{2},$
there exists a standard Wiener process $\left\{  \mathbf{W}\left(  s\right)
;\text{ }s\geq0\right\}  ,$\ defined on the probability space $\left(
\Omega,\mathcal{A},\mathbf{P}\right)  $ such that%
\begin{equation}
\sup_{x\geq1}x^{\left(  1/2-\epsilon\right)  /\gamma-1/\gamma_{2}}\left\vert
\sqrt{\upsilon}\mathbf{L}_{n}\left(  x;\upsilon\right)  -\mathcal{L}\left(
x;\mathbf{W}\right)  \right\vert \overset{\mathbf{P}}{\rightarrow}0,\text{ as
}N\rightarrow\infty, \label{approx}%
\end{equation}
where$\mathbb{\ }\left\{  \mathcal{L}\left(  x;\mathbf{W}\right)  ;\text{
}x>0\right\}  $\textbf{\ }is a Gaussian process defined by%
\begin{align*}
&  \frac{\gamma}{\gamma_{1}}x^{-1/\gamma_{1}}\left\{  x^{1/\gamma}%
\mathbf{W}\left(  x^{-1/\gamma}\right)  -\mathbf{W}\left(  1\right)  \right\}
\\
&  \ \ \ \ \ \ +\frac{\gamma}{\gamma_{1}+\gamma_{2}}x^{-1/\gamma_{1}}\int
_{0}^{1}s^{-\gamma/\gamma_{2}-1}\left\{  x^{1/\gamma}\mathbf{W}\left(
x^{-1/\gamma}s\right)  -\mathbf{W}\left(  s\right)  \right\}  ds.
\end{align*}
Let us decompose $\sqrt{\upsilon}\int_{1}^{\infty}\mathbf{L}_{n}\left(
x;\upsilon\right)  d\log^{\alpha}x$ into
\[
\int_{1}^{\infty}\mathcal{L}\left(  x;\mathbf{W}\right)  d\log^{\alpha}%
x+\int_{1}^{\infty}\left\{  \sqrt{\upsilon}\mathbf{L}_{n}\left(
x;\upsilon\right)  -\mathcal{L}\left(  x;\mathbf{W}\right)  \right\}
d\log^{\alpha}x.
\]
By using approximation $(\ref{approx}),$ we obtain $\int_{1}^{\infty}\left\{
\sqrt{\upsilon}\mathbf{L}_{n}\left(  x;\upsilon\right)  -\mathcal{L}\left(
x;\mathbf{W}\right)  \right\}  d\log^{\alpha}x=o_{\mathbf{P}}\left(  1\right)
.$ We showed in Lemma \ref{lemma2} that $\int_{1}^{\infty}\mathcal{L}\left(
x;\mathbf{W}\right)  d\log^{\alpha}x=O_{\mathbf{P}}\left(  1\right)  ,$
therefore $\int_{1}^{\infty}\mathbf{L}_{n}\left(  x;\upsilon\right)
d\log^{\alpha}x=O_{\mathbf{P}}\left(  \upsilon^{-1/2}\right)  ,$ it follows
that%
\begin{align}
M_{n}^{\left(  \alpha\right)  }\left(  \upsilon\right)   &  =\gamma
_{1}^{\alpha}\mu_{\alpha}^{\left(  1\right)  }+\upsilon^{-1/2}\int_{1}%
^{\infty}\mathcal{L}\left(  x;\mathbf{W}\right)  d\log^{\alpha}%
x\label{aproxi-M-alpha}\\
&  +\alpha\gamma_{1}^{\alpha-1}\mu_{\alpha}^{\left(  2\right)  }\left(
\rho_{1}\right)  \mathbf{A}_{0}\left(  n/\upsilon\right)  \left(
1+o_{\mathbf{p}}\left(  1\right)  \right)  +o_{\mathbf{P}}\left(
\upsilon^{-1/2}\right)  .\nonumber
\end{align}
Once again, by using the fact that $\int_{1}^{\infty}\mathcal{L}\left(
x;\mathbf{W}\right)  d\log^{\alpha}x=O_{\mathbf{P}}\left(  1\right)  ,$ we get%
\[
M_{n}^{\left(  \alpha\right)  }\left(  \upsilon\right)  =\gamma_{1}^{\alpha
}\mu_{\alpha}^{\left(  1\right)  }+\alpha\gamma_{1}^{\alpha-1}\mu_{\alpha
}^{\left(  2\right)  }\left(  \rho_{1}\right)  \mathbf{A}_{0}\left(
n/\upsilon\right)  \left(  1+o_{\mathbf{P}}\left(  1\right)  \right)
+o_{\mathbf{P}}\left(  \upsilon^{-1/2}\right)  .
\]
It particular, for $\alpha=1,$ we have $\mu_{1}^{\left(  1\right)  }=1,$ this
means that
\[
M_{n}^{\left(  1\right)  }\left(  \upsilon\right)  =\gamma_{1}+\mu
_{1}^{\left(  2\right)  }\left(  \rho_{1}\right)  \mathbf{A}_{0}\left(
n/\upsilon\right)  \left(  1+o_{\mathbf{p}}\left(  1\right)  \right)
+o_{\mathbf{P}}\left(  \upsilon^{-1/2}\right)  ,
\]
which implies that%
\begin{equation}
\left(  M_{n}^{\left(  1\right)  }\left(  \upsilon\right)  \right)
^{2}=\gamma_{1}^{2}+2\gamma_{1}\mu_{1}^{\left(  2\right)  }\left(  \rho
_{1}\right)  \mathbf{A}_{0}\left(  n/\upsilon\right)  \left(  1+o_{\mathbf{P}%
}\left(  1\right)  \right)  +o_{\mathbf{P}}\left(  \upsilon^{-1/2}\right)  .
\label{consi2-Malph}%
\end{equation}
Likewise, for $\alpha=2,$ we have $\mu_{2}^{\left(  1\right)  }=2,$ then
\begin{equation}
M_{n}^{\left(  2\right)  }\left(  \upsilon\right)  =2\gamma_{1}^{2}%
+2\gamma_{1}\mu_{2}^{\left(  2\right)  }\left(  \rho_{1}\right)
\mathbf{A}_{0}\left(  n/\upsilon\right)  \left(  1+o_{\mathbf{P}}\left(
1\right)  \right)  +o_{\mathbf{P}}\left(  \upsilon^{-1/2}\right)  .
\label{Mn2}%
\end{equation}
Similar to the definition of $M_{n}^{\left(  \alpha\right)  }\left(
\upsilon\right)  ,$ let $Q_{n}^{\left(  \alpha\right)  }\left(  \upsilon
\right)  $ be $Q^{\alpha}\left(  t;\mathbf{F}\right)  $ with $\mathbb{U}%
_{F^{\ast}}\left(  t\right)  $ and$\ \mathbf{F}$ respectively replaced by by
$X_{n-\upsilon:n}$ and $\mathbf{F}_{n}.$ From $\left(  \ref{Q}\right)  ,$ we
may write%
\[
Q_{n}^{\left(  \alpha\right)  }\left(  \upsilon\right)  =\frac{M_{n}^{\left(
\alpha\right)  }\left(  \upsilon\right)  -\Gamma\left(  \alpha+1\right)
\left(  M_{n}^{\left(  1\right)  }\left(  \upsilon\right)  \right)  ^{\alpha}%
}{M_{n}^{\left(  2\right)  }\left(  \upsilon\right)  -2\left(  M_{n}^{\left(
1\right)  }\left(  \upsilon\right)  \right)  ^{2}}.
\]
Then, by using the approximations above, we end up with%
\[
Q_{n}^{\left(  \alpha\right)  }\left(  \upsilon\right)  =\left(
1+o_{\mathbf{P}}\left(  1\right)  \right)  \frac{\alpha\gamma_{1}^{\alpha
-1}\left(  \mu_{\alpha}^{\left(  2\right)  }\left(  \rho_{1}\right)
-\mu_{\alpha}^{\left(  1\right)  }\mu_{1}^{\left(  2\right)  }\left(  \rho
_{1}\right)  \right)  }{2\gamma_{1}\left(  \mu_{2}^{\left(  2\right)  }\left(
\rho_{1}\right)  -\mu_{2}^{\left(  1\right)  }\mu_{1}^{\left(  2\right)
}\left(  \rho_{1}\right)  \right)  }.
\]
By replacing $\mu_{\alpha}^{\left(  1\right)  },$ $\mu_{1}^{\left(  1\right)
},$ $\mu_{\alpha}^{\left(  2\right)  }\left(  \rho_{1}\right)  $ and $\mu
_{1}^{\left(  2\right)  }\left(  \rho_{1}\right)  $ by their corresponding
expressions, given in $(\ref{mu1-2}),$ with the fact that $\alpha\Gamma\left(
\alpha\right)  =\Gamma\left(  \alpha+1\right)  ,$ we show that the previous
quotient equals $q_{\alpha}\left(  \rho_{1}\right)  $ given in $(\ref{q}).$
This implies that $Q_{n}^{\left(  \alpha\right)  }\left(  \upsilon\right)
\overset{\mathbf{P}}{\rightarrow}q_{\alpha}\left(  \rho_{1}\right)  $ and
therefore $S_{n}^{\left(  \alpha\right)  }\left(  \upsilon\right)
\overset{\mathbf{P}}{\rightarrow}s_{\alpha}\left(  \rho_{1}\right)  ,$ as
$N\rightarrow\infty,$ as well. By using the mean value theorem, we infer that
$\widehat{\rho}_{1}^{\left(  \alpha\right)  }=s_{\alpha}^{\leftarrow}\left(
S_{n}^{\left(  \alpha\right)  }\left(  \upsilon\right)  \right)
\overset{\mathbf{P}}{\rightarrow}\rho_{1},$ as sought. Let us now focus on the
asymptotic representation of $\widehat{\rho}_{1}^{\left(  \alpha\right)  }.$
We begin by denoting $\widetilde{M}_{n}^{\left(  \alpha\right)  }\left(
\upsilon\right)  ,$ $\widetilde{S}_{n}^{\left(  \alpha\right)  }\left(
\upsilon\right)  $ and $\widetilde{Q}_{n}^{\left(  \alpha\right)  }\left(
\upsilon\right)  $ the respective values of $M^{\left(  \alpha\right)
}\left(  t;\mathbf{F}\right)  ,$ $S^{\left(  \alpha\right)  }\left(
t;\mathbf{F}\right)  $ and $Q^{\left(  \alpha\right)  }\left(  t;\mathbf{F}%
\right)  $ when replacing $\mathbb{U}_{F^{\ast}}\left(  t\right)  $ by
$X_{n-\upsilon:n}.$ It is clear that the quantity $S_{n}^{\left(
\alpha\right)  }\left(  \upsilon\right)  -s_{\alpha}\left(  \rho_{1}\right)  $
may be decomposed into the sum of%
\[
T_{n1}:=-\delta\left(  \alpha\right)  \frac{\left(  Q_{n}^{\left(
\alpha+1\right)  }\left(  \upsilon\right)  \right)  ^{2}-\left(  \widetilde
{Q}_{n}^{\left(  \alpha+1\right)  }\left(  \upsilon\right)  \right)  ^{2}%
}{\left(  Q_{n}^{\left(  \alpha+1\right)  }\left(  \upsilon\right)
\widetilde{Q}_{n}^{\left(  \alpha+1\right)  }\left(  \upsilon\right)  \right)
^{2}}Q_{n}^{\left(  2\alpha\right)  }\left(  \upsilon;\mathbf{F}_{n}\right)
,
\]%
\[
T_{n2}:=\delta\left(  \alpha\right)  \frac{Q_{n}^{\left(  2\alpha\right)
}\left(  \upsilon\right)  -\widetilde{Q}_{n}^{\left(  2\alpha\right)  }\left(
\upsilon\right)  }{\left(  \widetilde{Q}_{n}^{\left(  \alpha+1\right)
}\left(  \upsilon\right)  \right)  ^{2}}\text{ and }T_{n3}:=\widetilde{S}%
_{n}^{\left(  \alpha\right)  }\left(  \upsilon\right)  -s_{\alpha}\left(
\rho_{1}\right)  .
\]
Since $Q_{n}^{\left(  \alpha\right)  }\left(  \upsilon\right)  \overset
{\mathbf{P}}{\rightarrow}q_{\alpha}\left(  \rho_{1}\right)  ,$ then by using
the mean value theorem, we get%
\[
T_{n1}=-\left(  1+o_{\mathbf{p}}\left(  1\right)  \right)  2\delta\left(
\alpha\right)  q_{2\alpha}q_{\alpha+1}^{-3}\left(  Q_{n}^{\left(
\alpha+1\right)  }\left(  \upsilon\right)  -\widetilde{Q}_{n}^{\left(
\alpha+1\right)  }\left(  \upsilon\right)  \right)  .
\]
Making use of the third-order condition $(\ref{Third}),$ with analogy of the
weak approximation given in \cite{GHP2003} (page 411), we write
\begin{align}
M_{n}^{\left(  \alpha\right)  }\left(  \upsilon\right)   &  =\gamma
_{1}^{\alpha}\mu_{\alpha}^{\left(  1\right)  }+\upsilon^{-1/2}\int_{1}%
^{\infty}\mathcal{L}\left(  x;\mathbf{W}\right)  d\log^{\alpha}x+\alpha
\gamma_{1}^{\alpha-1}\mu_{\alpha}^{\left(  2\right)  }\left(  \rho_{1}\right)
\mathbf{A}_{0}\left(  n/\upsilon\right) \label{A1}\\
&  +\alpha\gamma_{1}^{\alpha-1}\mu_{\alpha}^{\left(  4\right)  }\left(
\rho_{1},\beta_{1}\right)  \mathbf{A}_{0}\left(  n/\upsilon\right)
\mathbf{B}_{0}\left(  n/\upsilon\right)  \left(  1+o_{\mathbf{p}}\left(
1\right)  \right)  +o_{\mathbf{P}}\left(  \upsilon^{-1/2}\right)  .\nonumber
\end{align}
Since $\int_{1}^{\infty}\mathcal{L}\left(  x;\mathbf{W}\right)  d\log^{\alpha
}x=O_{\mathbf{P}}\left(  1\right)  ,$ then%
\begin{align}
M_{n}^{\left(  \alpha\right)  }\left(  \upsilon\right)   &  =\gamma
_{1}^{\alpha}\mu_{\alpha}^{\left(  1\right)  }+\alpha\gamma_{1}^{\alpha-1}%
\mu_{\alpha}^{\left(  2\right)  }\left(  \rho_{1}\right)  \mathbf{A}%
_{0}\left(  n/\upsilon\right) \label{A2}\\
&  \ +\alpha\gamma_{1}^{\alpha-1}\mu_{\alpha}^{\left(  4\right)  }\left(
\rho_{1},\beta_{1}\right)  \mathbf{A}_{0}\left(  n/\upsilon\right)
\mathbf{B}_{0}\left(  n/\upsilon\right)  \left(  1+o_{\mathbf{P}}\left(
1\right)  \right)  +o_{\mathbf{P}}\left(  \upsilon^{-1/2}\right)  .\nonumber
\end{align}
Let us write%
\begin{align*}
&  Q_{n}^{\left(  \alpha\right)  }\left(  \upsilon\right)  -\widetilde{Q}%
_{n}^{\left(  \alpha\right)  }\left(  \upsilon\right) \\
&  =\frac{M_{n}^{\left(  \alpha\right)  }\left(  \upsilon\right)
-\Gamma\left(  \alpha+1\right)  \left(  M_{n}^{\left(  1\right)  }\left(
\upsilon\right)  \right)  ^{2}}{M_{n}^{\left(  2\right)  }\left(
\upsilon\right)  -2\left(  M_{n}^{\left(  1\right)  }\left(  \upsilon\right)
\right)  ^{2}}-\frac{\widetilde{M}_{n}^{\left(  \alpha\right)  }\left(
\upsilon\right)  -\Gamma\left(  \alpha+1\right)  \left(  \widetilde{M}%
_{n}^{\left(  1\right)  }\left(  \upsilon\right)  \right)  ^{2}}{\widetilde
{M}_{n}^{\left(  2\right)  }\left(  \upsilon\right)  -2\left(  \widetilde
{M}_{n}^{\left(  1\right)  }\left(  \upsilon\right)  \right)  ^{2}}.
\end{align*}
By reducing to the common denominator and by using the weak approximations
$(\ref{A1})$ and $(\ref{A2})$ with the fact that $\mathbf{A}_{0}\left(
n/\upsilon\right)  \overset{\mathbf{P}}{\rightarrow}0,$ $\sqrt{\upsilon
}\mathbf{A}_{0}^{2}\left(  n/\upsilon\right)  $ and $\sqrt{\upsilon}%
\mathbf{A}_{0}\left(  n/\upsilon\right)  \mathbf{B}_{0}\left(  n/\upsilon
\right)  $ are stochastically bounded, we get%
\begin{align*}
&  \sqrt{\upsilon}\mathbf{A}_{0}\left(  n/\upsilon\right)  \left(
Q_{n}^{\left(  \alpha\right)  }\left(  \upsilon\right)  -\widetilde{Q}%
_{n}^{\left(  \alpha\right)  }\left(  \upsilon\right)  \right) \\
&  =\int_{1}^{\infty}\mathcal{L}\left(  x;\mathbf{W}\right)  dg_{1}%
(x;\alpha)+\theta_{1}\left(  \alpha\right)  \sqrt{\upsilon}\mathbf{A}%
_{0}\left(  n/\upsilon\right)  \mathbf{B}_{0}\left(  n/\upsilon\right)
+o_{\mathbf{P}}\left(  1\right)  ,
\end{align*}
where
\[
g_{1}(x;\alpha):=\frac{\gamma_{1}^{\alpha-1}}{2m_{2}}\left\{  \gamma
_{1}^{-\alpha}\log^{\alpha}x-\frac{\alpha\Gamma\left(  \alpha\right)  }%
{2}r_{\alpha}\gamma_{1}^{-2}\log^{2}x-\left(  \alpha\mu_{\alpha}^{\left(
1\right)  }-2\alpha\Gamma\left(  \alpha\right)  r_{\alpha}\right)  \gamma
_{1}^{-1}\log x\right\}  ,
\]
and $\theta_{1}\left(  \alpha\right)  :=\alpha\gamma_{1}^{\alpha-2}\left\{
d_{\alpha}-\Gamma\left(  \alpha\right)  r_{\alpha}d_{2}\right\}  /\left(
2m_{2}\right)  $ with $d_{\alpha},r_{\alpha}$ and $m_{2}$ being those defined
in the beginning of Section \ref{sec2}. It follows that%
\begin{align*}
&  \sqrt{\upsilon}\mathbf{A}_{0}\left(  n/\upsilon\right)  T_{n1}\\
&  =-2\delta\left(  \alpha\right)  q_{2\alpha}q_{\alpha+1}^{-3}\left\{
\int_{1}^{\infty}\mathcal{L}\left(  x;\mathbf{W}\right)  dg_{1}(x;\alpha
+1)+\theta_{1}\left(  \alpha+1\right)  \sqrt{\upsilon}\mathbf{A}_{0}\left(
n/\upsilon\right)  \mathbf{B}_{0}\left(  n/\upsilon\right)  +o_{\mathbf{P}%
}\left(  1\right)  \right\}  .
\end{align*}
Likewise, by similar arguments, we also get%
\begin{align*}
&  \sqrt{\upsilon}\mathbf{A}_{0}\left(  n/\upsilon\right)  T_{n2}\\
&  =\delta\left(  \alpha\right)  q_{\alpha+1}^{-2}\left\{  \int_{1}^{\infty
}\mathcal{L}\left(  x;\mathbf{W}\right)  dg_{1}(x;2\alpha)+\theta_{1}\left(
2\alpha\right)  \sqrt{\upsilon}\mathbf{A}_{0}\left(  n/\upsilon\right)
\mathbf{B}_{0}\left(  n/\upsilon\right)  +o_{\mathbf{P}}\left(  1\right)
\right\}  .
\end{align*}
Therefore%
\[
\sqrt{\upsilon}\mathbf{A}_{0}\left(  n/\upsilon\right)  \left(  T_{n1}%
+T_{n2}\right)  =\int_{1}^{\infty}\mathcal{L}\left(  x;\mathbf{W}\right)
dg(x;\alpha)+K\left(  \alpha\right)  \sqrt{\upsilon}\mathbf{A}_{0}\left(
n/\upsilon\right)  \mathbf{B}_{0}\left(  n/\upsilon\right)  +o_{\mathbf{P}%
}\left(  1\right)  ,
\]
where $K\left(  \alpha\right)  :=\delta\left(  \alpha\right)  \left(
q_{\alpha+1}^{-2}\theta_{1}\left(  2\alpha\right)  -2q_{2\alpha}q_{\alpha
+1}^{-3}\theta_{1}\left(  \alpha+1\right)  \right)  $ and%
\[
g(x;\alpha):=\delta\left(  \alpha\right)  \left(  q_{\alpha+1}^{-2}%
g_{1}\left(  x;2\alpha\right)  -2q_{2\alpha}q_{\alpha+1}^{-3}g_{1}\left(
x;\alpha+1\right)  \right)  .
\]
Once again by using the third-order condition $(\ref{Third})$ with the fact
that $\mathbf{A}_{0}\left(  n/\upsilon\right)  \overset{\mathbf{P}%
}{\rightarrow}0$ and $\sqrt{\upsilon}\mathbf{A}_{0}\left(  n/\upsilon\right)
\mathbf{B}_{0}\left(  n/\upsilon\right)  =O_{\mathbf{P}}\left(  1\right)  ,$
we show that $\sqrt{\upsilon}\mathbf{A}_{0}\left(  n/\upsilon\right)
T_{n3}=\eta_{1}\sqrt{\upsilon}\mathbf{A}_{0}^{2}\left(  n/\upsilon\right)
+o_{\mathbf{P}}\left(  1\right)  .$ It is easy to check that $K\left(
\alpha\right)  \equiv\eta_{2},$ hence we have%
\begin{align*}
&  \sqrt{\upsilon}\mathbf{A}_{0}\left(  n/\upsilon\right)  \left(
S_{n}^{\left(  \alpha\right)  }\left(  \upsilon\right)  -s_{\alpha}\left(
\rho_{1}\right)  \right) \\
&  =\int_{1}^{\infty}\mathcal{L}\left(  x;\mathbf{W}\right)  dg(x;\alpha
)+\eta_{1}\sqrt{\upsilon}\mathbf{A}_{0}^{2}\left(  n/\upsilon\right)
+\eta_{2}\sqrt{\upsilon}\mathbf{A}_{0}\left(  n/\upsilon\right)
\mathbf{B}_{0}\left(  n/\upsilon\right)  +o_{\mathbf{P}}\left(  1\right)  ,
\end{align*}
where $\eta_{1}$ and $\eta_{2}$ are those defined in the beginning of Section
\ref{sec2}. Recall that $S_{n}^{\left(  \alpha\right)  }\left(  \upsilon
\right)  =s_{\alpha}\left(  \widehat{\rho}_{1}^{\left(  \alpha\right)
}\right)  ,$ then in view of the mean value theorem and the consistency of
$\widehat{\rho}_{1}^{\left(  \alpha\right)  },$ we end up with
\begin{align*}
&  s_{\alpha}^{\prime}\left(  \rho_{1}\right)  \sqrt{\upsilon}\mathbf{A}%
_{0}\left(  n/\upsilon\right)  \left(  \widehat{\rho}_{1}^{\left(
\alpha\right)  }-\rho_{1}\right) \\
&  =\int_{1}^{\infty}\mathcal{L}\left(  x;\mathbf{W}\right)  dg(x;\alpha
)+\eta_{1}\sqrt{\upsilon}\mathbf{A}_{0}^{2}\left(  n/\upsilon\right)
+\eta_{2}\sqrt{\upsilon}\mathbf{A}_{0}\left(  n/\upsilon\right)
\mathbf{B}_{0}\left(  n/\upsilon\right)  +o_{\mathbf{P}}\left(  1\right)  .
\end{align*}
Finally, integrating by parts with elementary calculations complete the proof
of the second part of the theorem, namely the Gaussian approximation of
$\widehat{\rho}_{1}^{\left(  \alpha\right)  }.$ For the third assertion, it
suffices to calculate
\[
\mathbf{E}\left[  \int_{0}^{1}s^{-\gamma/\gamma_{2}-1}\Delta_{\alpha
}(s)\mathbf{W}\left(  s\right)  ds-\xi\mathbf{W}\left(  1\right)  \right]
^{2}%
\]
to get the asymptotic variance $\sigma_{\alpha}^{2},$ therefore we omit details.

\subsection{\textbf{Proof of Theorem }$\ref{Theorem2}$}

Let us write%
\[
\sqrt{k}\left(  \widehat{\gamma}_{1}-\gamma_{1}\right)  =\sqrt{k}\left(
M_{n}^{\left(  1\right)  }\left(  k\right)  -\gamma_{1}\right)  +\frac
{\widehat{\rho}_{1}^{\left(  \ast\right)  }-1}{2\widehat{\rho}_{1}^{\left(
\ast\right)  }M_{n}^{\left(  1\right)  }\left(  k\right)  }\sqrt{k}\left(
M_{n}^{\left(  2\right)  }\left(  k\right)  -2\left(  M_{n}^{\left(  1\right)
}\left(  k\right)  \right)  ^{2}\right)  .
\]
From, Theorem 3.1 in \cite{BchMN-16a} and Theorem \ref{Theorem1} above
both\ $M_{n}^{\left(  1\right)  }\left(  k\right)  =\widehat{\gamma}%
_{1}^{\left(  BMN\right)  }$ and $\widehat{\rho}_{1}^{\left(  \ast\right)  }$
are consistent for $\gamma_{1}$ and $\rho_{1}$ respectively. It follows that
\begin{align*}
&  \sqrt{k}\left(  \widehat{\gamma}_{1}-\gamma_{1}\right) \\
&  =\sqrt{k}\left(  M_{n}^{\left(  1\right)  }\left(  k\right)  -\gamma
_{1}\right)  +\frac{\rho_{1}-1}{2\gamma_{1}\rho_{1}}\sqrt{k}\left(
M_{n}^{\left(  2\right)  }\left(  k\right)  -2\left(  M_{n}^{\left(  1\right)
}\left(  k\right)  \right)  ^{2}\right)  \left(  1+o_{\mathbf{P}}\left(
1\right)  \right)  .
\end{align*}
By applying the weak approximation $\left(  \ref{aproxi-M-alpha}\right)  ,$
for $\alpha=1,$ we get%
\begin{equation}
\sqrt{k}\left(  M_{n}^{\left(  1\right)  }\left(  k\right)  -\gamma
_{1}\right)  =\int_{1}^{\infty}\mathcal{L}\left(  x;\mathbf{W}\right)  d\log
x+\frac{\sqrt{k}\mathbf{A}_{0}\left(  n/k\right)  }{1-\rho_{1}}+o_{\mathbf{P}%
}\left(  1\right)  . \label{aa}%
\end{equation}
Using the mean value theorem and the consistency of $M_{n}^{\left(  1\right)
}\left(  k\right)  $ yield%
\begin{align*}
&  \sqrt{k}\left(  \left(  M_{n}^{\left(  1\right)  }\left(  k\right)
\right)  ^{2}-\gamma_{1}^{2}\right) \\
&  =2\gamma_{1}\left\{  \int_{1}^{\infty}\mathcal{L}\left(  x;\mathbf{W}%
\right)  d\log x+\frac{\sqrt{k}\mathbf{A}_{0}\left(  n/k\right)  }{1-\rho_{1}%
}+o_{\mathbf{P}}\left(  1\right)  \right\}  \left(  1+o_{\mathbf{P}}\left(
1\right)  \right)  .
\end{align*}
From Lemma $\ref{lemma2}$ and the assumption $\sqrt{k}\mathbf{A}_{0}\left(
n/k\right)  =O_{\mathbf{P}}\left(  1\right)  $ as $N\rightarrow\infty$ we have%
\[
\sqrt{k}\left(  \left(  M_{n}^{\left(  1\right)  }\left(  k\right)  \right)
^{2}-\gamma_{1}^{2}\right)  =\int_{1}^{\infty}\mathcal{L}\left(
x;\mathbf{W}\right)  d\left(  2\gamma_{1}\log x\right)  +\frac{2\gamma_{1}%
}{1-\rho_{1}}\sqrt{k}\mathbf{A}_{0}\left(  n/k\right)  +o_{\mathbf{P}}\left(
1\right)  .
\]
Once again, by applying the weak approximation $(\ref{aproxi-M-alpha}$ $,$ for
$\alpha=2,$ we write%
\[
\sqrt{k}\left(  M_{n}^{\left(  2\right)  }\left(  k\right)  -2\gamma_{1}%
^{2}\right)  =\int_{1}^{\infty}\mathcal{L}\left(  x;\mathbf{W}\right)
d\log^{2}x+2\gamma_{1}\mu_{2}^{\left(  2\right)  }\left(  \rho_{1}\right)
\sqrt{k}\mathbf{A}_{0}\left(  n/k\right)  +o_{\mathbf{P}}\left(  1\right)  ,
\]
where $\mu_{2}^{\left(  2\right)  }\left(  \rho_{1}\right)  =\left(  1-\left(
1-\rho_{1}\right)  ^{2}\right)  /\left(  \rho_{1}\left(  1-\rho_{1}\right)
^{2}\right)  .$ It follows that%
\begin{align}
&  \sqrt{k}\left(  M_{n}^{\left(  2\right)  }\left(  k\right)  -2\left(
M_{n}^{\left(  1\right)  }\left(  k\right)  \right)  ^{2}\right) \label{bb}\\
&  =\int_{1}^{\infty}\mathcal{L}\left(  x;\mathbf{W}\right)  d\left(  \log
^{2}x-4\gamma_{1}\log x\right)  +\frac{2\gamma_{1}\rho_{1}}{\left(  1-\rho
_{1}\right)  ^{2}}\sqrt{k}\mathbf{A}_{0}\left(  n/k\right)  +o_{\mathbf{P}%
}\left(  1\right)  .\nonumber
\end{align}
By combining approximations $(\ref{aa})$ and $(\ref{bb}),$ we obtain
\[
\sqrt{k}\left(  \widehat{\gamma}_{1}-\gamma_{1}\right)  =\int_{1}^{\infty
}\mathcal{L}\left(  x;\mathbf{W}\right)  d\Psi\left(  x\right)  +o_{\mathbf{P}%
}\left(  \sqrt{k}\mathbf{A}_{0}\left(  n/k\right)  \right)  ,\text{ as
}N\rightarrow\infty,
\]
where $\Psi\left(  x\right)  :=\tau_{6}\log x+\tau_{5}\log^{2}x.$ Finally,
making an integration by parts and a change of variables, with elementary
calculations, achieve the proof of the first assertion of the theorem. The
second part is straightforward.

\section{\textbf{Appendix\label{Appendix}}}

\begin{lemma}
\textbf{\label{lemma1}}Assume that the second-order regular variation
condition $(\ref{second-order})$ holds, then for any $\alpha>0$%
\[%
\begin{tabular}
[c]{l}%
$\left(  i\right)  $ $\lim_{t\rightarrow\infty}\dfrac{M^{\left(
\alpha\right)  }\left(  t;\mathbf{F}\right)  -\mu_{\alpha}^{\left(  1\right)
}\left(  M^{\left(  1\right)  }\left(  t;\mathbf{F}\right)  \right)  ^{\alpha
}}{\left(  M^{\left(  1\right)  }\left(  t;\mathbf{F}\right)  \right)
^{\alpha-1}\mathbf{A}_{0}\left(  t\right)  }=\alpha\left(  \mu_{\alpha
}^{\left(  2\right)  }\left(  \rho_{1}\right)  -\mu_{\alpha}^{\left(
1\right)  }\mu_{1}^{\left(  2\right)  }\left(  \rho_{1}\right)  \right)
.\medskip$\\
$\left(  ii\right)  $ $\lim_{t\rightarrow\infty}Q^{\left(  \alpha\right)
}\left(  t;\mathbf{F}\right)  =q_{\alpha}\left(  \rho_{1}\right)  .\medskip$\\
$\left(  iii\right)  \text{ }\lim_{t\rightarrow\infty}S^{\left(
\alpha\right)  }\left(  t;\mathbf{F}\right)  =s_{\alpha}\left(  \rho
_{1}\right)  .$%
\end{tabular}
\]

\end{lemma}

\begin{proof}
Let us consider assertion $\left(  i\right)  .$ We begin by letting
\[
U^{\left(  \alpha\right)  }:=-\int_{1}^{\infty}\log^{\alpha}sds^{-1/\gamma
_{1}}\text{ and }\ell\left(  t\right)  :=\frac{M^{\left(  \alpha\right)
}\left(  t;\mathbf{F}\right)  -\mu_{\alpha}^{\left(  1\right)  }\left(
M^{\left(  1\right)  }\left(  t;\mathbf{F}\right)  \right)  ^{\alpha}%
}{\mathbf{A}_{0}\left(  t\right)  },
\]
to show that, for any $\alpha>0$
\begin{equation}
\lim_{t\rightarrow\infty}\ell\left(  t\right)  =\alpha\gamma_{1}^{\alpha
-1}\left(  \mu_{\alpha}^{\left(  2\right)  }\left(  \rho_{1}\right)
-\mu_{\alpha}^{\left(  1\right)  }\mu_{1}^{\left(  2\right)  }\left(  \rho
_{1}\right)  \right)  . \label{limit1}%
\end{equation}
Indeed, let us first decompose $\ell\left(  t\right)  $ into%
\[
\frac{M^{\left(  \alpha\right)  }\left(  t;\mathbf{F}\right)  -U^{\left(
\alpha\right)  }}{\mathbf{A}_{0}\left(  t\right)  }-\mu_{\alpha}^{\left(
1\right)  }\frac{\left(  M^{\left(  1\right)  }\left(  t;\mathbf{F}\right)
\right)  ^{\alpha}-\left(  U^{\left(  1\right)  }\right)  ^{\alpha}%
}{\mathbf{A}_{0}\left(  t\right)  }+\frac{U^{\left(  \alpha\right)  }\left(
t\right)  -\mu_{\alpha}^{\left(  1\right)  }\left(  U^{\left(  1\right)
}\right)  ^{\alpha}}{\mathbf{A}_{0}\left(  t\right)  },
\]
and note that $\mu_{\alpha}^{\left(  1\right)  }=\Gamma\left(  \alpha
+1\right)  $ where $\Gamma\left(  a\right)  =\int_{0}^{\infty}e^{-x}%
x^{a-1}dx=\int_{1}^{\infty}t^{-2}\log^{a-1}tdt,$ $a>0.$ It is easy to verify
that $U^{\left(  \alpha\right)  }-\mu_{\alpha}^{\left(  1\right)  }\left(
U^{\left(  1\right)  }\right)  ^{\alpha}=0,$ therefore%
\begin{equation}
\ell\left(  t\right)  =\frac{M^{\left(  \alpha\right)  }\left(  t;\mathbf{F}%
\right)  -U^{\left(  \alpha\right)  }}{\mathbf{A}_{0}\left(  t\right)  }%
-\mu_{\alpha}^{\left(  1\right)  }\frac{\left(  M^{\left(  1\right)  }\left(
t;\mathbf{F}\right)  \right)  ^{\alpha}-\left(  U^{\left(  1\right)  }\right)
^{\alpha}}{\mathbf{A}_{0}\left(  t\right)  }. \label{teta}%
\end{equation}
Recall that $M^{\left(  \alpha\right)  }\left(  t;\mathbf{F}\right)  =\int
_{u}^{\infty}\log^{\alpha}\left(  x/u\right)  d\mathbf{F}\left(  x\right)
/\overline{\mathbf{F}}\left(  u\right)  ,$ where $u:=\mathbb{U}_{F^{\ast}%
}\left(  t\right)  ,$ which by a change of variables and an integration by
parts, may be rewritten into
\[
\int_{1}^{\infty}\frac{\overline{\mathbf{F}}\left(  ux\right)  }%
{\overline{\mathbf{F}}\left(  u\right)  }d\log^{\alpha}x=:M_{u}^{\left(
\alpha\right)  }\left(  \mathbf{F}\right)  ,
\]
Making use, once again, of Proposition 4 of \cite{Hua}, we write: for possibly
different function $\widetilde{\mathbf{A}}_{\mathbf{F}},$ with $\widetilde
{\mathbf{A}}_{\mathbf{F}}\left(  y\right)  \sim\mathbf{A}_{\mathbf{F}}\left(
y\right)  ,$ as $y\rightarrow\infty,$ for any $0<\epsilon<1$ and $x\geq1,$ we
have%
\[
\left\vert \frac{\overline{\mathbf{F}}\left(  ux\right)  /\overline
{\mathbf{F}}\left(  u\right)  -x^{-1/\gamma_{1}}}{\gamma_{1}^{-2}%
\widetilde{\mathbf{A}}_{\mathbf{F}}\left(  1/\overline{\mathbf{F}}\left(
u\right)  \right)  }-x^{-1/\gamma_{1}}\dfrac{x^{\rho_{1}/\gamma_{1}}-1}%
{\gamma_{1}/\rho_{1}}\right\vert \leq\epsilon x^{-1/\gamma_{1}+\epsilon
},\text{ as }u\rightarrow\infty.
\]
By using elementary analysis, we easily show that this inequality implies that%
\[
\frac{M_{u}^{\left(  \alpha\right)  }\left(  \mathbf{F}\right)  -U^{\left(
\alpha\right)  }}{\widetilde{\mathbf{A}}_{\mathbf{F}}\left(  1/\overline
{\mathbf{F}}\left(  u\right)  \right)  }\rightarrow\alpha\gamma_{1}^{\alpha
-1}\mu_{\alpha}^{\left(  2\right)  }\left(  \rho_{1}\right)  ,\text{ as
}u\rightarrow\infty.
\]
Hence, since $1/\overline{\mathbf{F}}\left(  u\right)  \rightarrow\infty$ as
$u\rightarrow\infty,$ then $\widetilde{\mathbf{A}}_{\mathbf{F}}\left(
1/\overline{\mathbf{F}}\left(  u\right)  \right)  \sim\mathbf{A}_{\mathbf{F}%
}\left(  1/\overline{\mathbf{F}}\left(  u\right)  \right)  =\mathbf{A}%
_{0}\left(  t\right)  .$ This means that
\begin{equation}
\frac{M^{\left(  \alpha\right)  }\left(  t,\mathbf{F}\right)  -U^{\left(
\alpha\right)  }}{\mathbf{A}_{0}\left(  t\right)  }\rightarrow\alpha\gamma
_{1}^{\alpha-1}\mu_{\alpha}^{\left(  2\right)  }\left(  \rho_{1}\right)
,\text{ as }t\rightarrow\infty. \label{a}%
\end{equation}
Note that for $\alpha=1,$ we have $U^{\left(  1\right)  }=\gamma_{1}$ and
therefore $\left(  M^{\left(  1\right)  }\left(  t,\mathbf{F}\right)
-\gamma_{1}\right)  /\mathbf{A}_{0}\left(  t\right)  \rightarrow\mu
_{1}^{\left(  2\right)  }\left(  \rho_{1}\right)  ,$ which implies that
$M^{\left(  1\right)  }\left(  t,\mathbf{F}\right)  \rightarrow\gamma_{1}.$ By
using the mean value theorem and the previous two results we get%
\begin{equation}
\frac{\left(  M^{\left(  1\right)  }\left(  t,\mathbf{F}\right)  \right)
^{\alpha}-\gamma_{1}^{\alpha}}{\mathbf{A}_{0}\left(  t\right)  }%
\rightarrow\alpha\gamma_{1}^{\alpha-1}\mu_{1}^{\left(  2\right)  }\left(
\rho_{1}\right)  ,\text{ as }t\rightarrow\infty. \label{b}%
\end{equation}
Combining $\left(  \ref{teta}\right)  ,$ $\left(  \ref{a}\right)  $ and
$\left(  \ref{b}\right)  $ leads to $\left(  \ref{limit1}\right)  .$ Finally,
we use the fact that $M^{\left(  1\right)  }\left(  t,\mathbf{F}\right)
\rightarrow\gamma_{1}$ to achieve the proof of assertion $\left(  i\right)  .$
To show assertion $\left(  ii\right)  ,$ we apply assertion $\left(  i\right)
$ twice, for $\alpha>0$ and for $\alpha=2,$ then we divide the respective
results to get%
\begin{align*}
Q^{\left(  \alpha\right)  }\left(  t;\mathbf{F}\right)   &  =\frac{M^{\left(
\alpha\right)  }\left(  t;\mathbf{F}\right)  -\mu_{\alpha}^{\left(  1\right)
}\left[  M^{\left(  1\right)  }\left(  t;\mathbf{F}\right)  \right]  ^{\alpha
}}{M^{\left(  2\right)  }\left(  t;\mathbf{F}\right)  -2\left[  M^{\left(
1\right)  }\left(  t;\mathbf{F}\right)  \right]  ^{2}}\\
&  \sim\frac{\left(  M^{\left(  1\right)  }\left(  t;\mathbf{F}\right)
\right)  ^{\alpha-1}}{M^{\left(  1\right)  }\left(  t;\mathbf{F}\right)
}\frac{\alpha\left(  \mu_{\alpha}^{\left(  2\right)  }\left(  \rho_{1}\right)
-\mu_{\alpha}^{\left(  1\right)  }\mu_{1}^{\left(  2\right)  }\left(  \rho
_{1}\right)  \right)  }{2\left(  \mu_{2}^{\left(  2\right)  }\left(  \rho
_{1}\right)  -\mu_{2}^{\left(  1\right)  }\mu_{1}^{\left(  2\right)  }\left(
\rho_{1}\right)  \right)  }.
\end{align*}
By replacing $\mu_{\alpha}^{\left(  1\right)  }$ and $\mu_{\alpha}^{\left(
2\right)  }\left(  \rho_{1}\right)  $ by their expressions, given in
$(\ref{mu1-2}),$ we get%
\begin{align*}
&  \alpha\left(  \mu_{\alpha}^{\left(  2\right)  }\left(  \rho_{1}\right)
-\mu_{\alpha}^{\left(  1\right)  }\mu_{1}^{\left(  2\right)  }\left(  \rho
_{1}\right)  \right) \\
&  =\alpha\left\{  \frac{\Gamma\left(  \alpha\right)  \left(  1-\left(
1-\rho_{1}\right)  ^{\alpha}\right)  }{\rho_{1}\left(  1-\rho_{1}\right)
^{\alpha}}-\Gamma\left(  \alpha+1\right)  \frac{\Gamma\left(  1\right)
\left(  1-\left(  1-\rho_{1}\right)  \right)  }{\rho_{1}\left(  1-\rho
_{1}\right)  }\right\} \\
&  =\alpha\left\{  \frac{\Gamma\left(  \alpha\right)  \left(  1-\left(
1-\rho_{1}\right)  ^{\alpha}\right)  }{\rho_{1}\left(  1-\rho_{1}\right)
^{\alpha}}-\frac{\Gamma\left(  \alpha+1\right)  }{1-\rho_{1}}\right\}  .
\end{align*}
Since $M^{\left(  1\right)  }\left(  t,\mathbf{F}\right)  \rightarrow
\gamma_{1},$ then%
\[
Q^{\left(  \alpha\right)  }\left(  t;\mathbf{F}\right)  \rightarrow
\frac{\gamma_{1}^{\alpha-2}\Gamma\left(  \alpha+1\right)  \left(  1-\left(
1-\rho_{1}\right)  ^{\alpha}-\alpha\rho_{1}\left(  1-\rho_{1}\right)
^{\alpha-1}\right)  }{\rho_{1}^{2}\left(  1-\rho_{1}\right)  ^{\alpha-2}%
},\text{ as }t\rightarrow\infty,
\]
which is $q_{\alpha}\left(  \rho_{1}\right)  $ given in $\left(
\ref{q}\right)  .$ For assertion $\left(  iii\right)  ,$ it is clear that%
\[
\delta\left(  \alpha\right)  \frac{Q_{t}^{\left(  2\alpha\right)  }}{\left(
Q_{t}^{\left(  \alpha+1\right)  }\right)  ^{2}}\rightarrow\frac{\rho_{1}%
^{2}\left(  1-\left(  1-\rho_{1}\right)  ^{2\alpha}-2\alpha\rho_{1}\left(
1-\rho_{1}\right)  ^{2\alpha-1}\right)  }{\left(  1-\left(  1-\rho_{1}\right)
^{\alpha+1}-\left(  \alpha+1\right)  \rho_{1}\left(  1-\rho_{1}\right)
^{\alpha}\right)  ^{2}},
\]
which meets the expression of $s_{\alpha}\left(  \rho_{1}\right)  $ given in
$\left(  \ref{s}\right)  .$
\end{proof}

\begin{lemma}
\textbf{\label{lemma2} }For any\textbf{\ }$\alpha>0,$ we have $\int
_{1}^{\infty}\mathcal{L}\left(  x;\mathbf{W}\right)  d\log^{\alpha
}x=O_{\mathbf{P}}\left(  1\right)  .$
\end{lemma}

\begin{proof}
Observe that $\int_{1}^{\infty}\mathcal{L}\left(  x;\mathbf{W}\right)
d\log^{\alpha}x$ may be decomposed into the sum of%
\[
I_{1}:=\frac{\gamma}{\gamma_{1}}\int_{1}^{\infty}x^{1/\gamma_{2}}%
\mathbf{W}\left(  x^{-1/\gamma}\right)  d\log^{\alpha}x,\text{ }I_{2}%
:=-\frac{\gamma}{\gamma_{1}}\mathbf{W}\left(  1\right)  \int_{1}^{\infty
}x^{-1/\gamma_{1}}d\log^{\alpha}x,
\]%
\[
I_{3}:=\frac{\gamma}{\gamma_{1}+\gamma_{2}}\int_{1}^{\infty}x^{1/\gamma_{2}%
}\left\{  \int_{0}^{1}s^{-\gamma/\gamma_{2}-1}\mathbf{W}\left(  x^{-1/\gamma
}s\right)  ds\right\}  d\log^{\alpha}x,
\]
and%
\[
I_{4}:=-\frac{\gamma}{\gamma_{1}+\gamma_{2}}\left\{  \int_{0}^{1}%
s^{-\gamma/\gamma_{2}-1}\mathbf{W}\left(  s\right)  ds\right\}  \int
_{1}^{\infty}x^{-1/\gamma_{1}}d\log^{\alpha}x.
\]
Next we show that $I_{i}=O_{\mathbf{P}}\left(  1\right)  ,$ $i=1,...,4.$ To
this end, we will show that $\mathbf{E}\left\vert I_{i}\right\vert $ is finite
for $i=1,...,4.$ Indeed, we have%
\[
\mathbf{E}\left\vert I_{1}\right\vert \leq\left(  \gamma/\gamma_{1}\right)
\int_{1}^{\infty}x^{-1/\gamma_{1}}x^{1/\gamma}\mathbf{E}\left\vert
\mathbf{W}\left(  x^{-1/\gamma}\right)  \right\vert d\log^{\alpha}x.
\]
Since $\mathbf{E}\left\vert \mathbf{W}\left(  y\right)  \right\vert \leq
\sqrt{y},$ for any $0\leq y\leq1,$ then $\mathbf{E}\left\vert I_{1}\right\vert
\leq\left(  \alpha\gamma/\gamma_{1}\right)  \int_{1}^{\infty}x^{1/\gamma
_{2}-1/\left(  2\gamma\right)  -1}\log^{\alpha-1}xdx.$ By successively making
two changes of variables $\log x=t,$ then $\left(  -1/\gamma_{2}+1/\left(
2\gamma\right)  +1\right)  t=s,$ we end up with $\mathbf{E}\left\vert
I_{1}\right\vert \leq\gamma\gamma_{1}^{\alpha-1}\left(  2\gamma/\left(
2\gamma-\gamma_{1}\right)  \right)  ^{\alpha}\Gamma\left(  \alpha+1\right)  $
which is finite for any $\alpha>0.$ By similar arguments we also show that
$\mathbf{E}\left\vert I_{2}\right\vert \leq\gamma\gamma_{1}^{\alpha-1}%
\Gamma\left(  \alpha+1\right)  $ which is finite as well. For the third term
$I_{3},$ we have%
\[
\mathbf{E}\left\vert I_{3}\right\vert \leq\frac{\gamma}{\gamma_{1}+\gamma_{2}%
}\int_{1}^{\infty}x^{1/\gamma_{2}}\left\{  \int_{0}^{1}s^{-\gamma/\gamma
_{2}-1}\mathbf{E}\left\vert \mathbf{W}\left(  x^{-1/\gamma}s\right)
\right\vert ds\right\}  d\log^{\alpha}x.
\]
By elementary calculations, we get%
\[
\mathbf{E}\left\vert I_{3}\right\vert \leq\frac{\gamma_{2}\left(
2\gamma\right)  ^{\alpha+1}}{\left(  \gamma_{2}-2\gamma\right)  \left(
\gamma_{1}+\gamma_{2}\right)  }\left(  \frac{\gamma_{1}}{2\gamma-\gamma_{1}%
}\right)  ^{\alpha}\Gamma\left(  \alpha+1\right)  ,
\]
which is also finite. By using similar arguments, we get
\[
\mathbf{E}\left\vert I_{4}\right\vert \leq\frac{2\gamma_{2}\gamma\gamma
_{1}^{\alpha}\Gamma\left(  \alpha+1\right)  }{\left(  \gamma_{1}+\gamma
_{2}\right)  \left(  \gamma_{2}-2\gamma\right)  }<\infty,
\]
as sought.
\end{proof}

\end{document}